\documentclass[reqno, 11pt, letterpaper]{amsart}

\usepackage{amsmath}
\usepackage{enumitem}
\usepackage{amsfonts}
\usepackage{amssymb}
\usepackage{graphicx}
\usepackage{amsthm,graphicx,color,yfonts}
\usepackage{pgfplots}
\usepackage{comment}
\pgfplotsset{compat=1.10}
\usepgfplotslibrary{fillbetween}
\usepackage{hyperref}
\usepackage[utf8]{inputenc}
\usepackage{amsmath}
\usepackage{calc}

\usepackage{accents}

\newtheorem{theorem}{Theorem}

\newtheorem{proposition}[theorem]{Proposition}
\newtheorem{lemma}[theorem]{Lemma}

\theoremstyle{remark}
\newtheorem{remark}[theorem]{Remark}

\newcommand{\ls}{\lesssim}
\newcommand{\gs}{\gtrsim}
\newcommand{\la}{\langle}
\newcommand{\ra}{\rangle}
\newcommand{\R}{\mathbb{R}}

\newcommand{\C}{\mathbb{C}}
\newcommand{\Z}{\mathbb{Z}}
\newcommand{\N}{\mathbb{N}}

\newcommand{\ep}{\epsilon}

\newcommand{\He}{\textbf{\textup{H}}}

\usepackage{wrapfig}
\usepackage{tikz}
\usetikzlibrary{arrows,calc,decorations.pathreplacing}
\definecolor{light-gray1}{gray}{0.90}
\definecolor{light-gray2}{gray}{0.80}
\definecolor{light-gray3}{gray}{0.60}

\numberwithin{equation}{section}

\numberwithin{theorem}{section}

\numberwithin{table}{section}

\numberwithin{figure}{section}

\ifx\pdfoutput\undefined
  \DeclareGraphicsExtensions{.pstex, .eps}
\else
  \ifx\pdfoutput\relax
    \DeclareGraphicsExtensions{.pstex, .eps}
  \else
    \ifnum\pdfoutput>0
      \DeclareGraphicsExtensions{.pdf}
    \else
      \DeclareGraphicsExtensions{.pstex, .eps}
    \fi
  \fi
\fi

\title[Strichartz estimates for Schrödinger equations and application]{Strichartz estimates for higher-order Schrödinger equations and their applications}

\author[Y. Hong]{Younghun Hong}
\address{Department of Mathematics, Chung-Ang University, Seoul 06974, Korea}
\email{yhhong@cau.ac.kr}

\author[C. Kwak]{Chulkwang Kwak}
\address{Department of Mathematics, Ewha Womans University, Seoul 03760, Korea}
\email{ckkwak@ewha.ac.kr}

\author[C. Yang]{Changhun Yang}
\address{Department of Mathematics, Chungbuk national University, Cheongju 28644 , Korea}
\email{chyang@chungbuk.ac.kr}

\begin{document}

\maketitle


\begin{abstract}
We consider the higher-order linear Schr\"odinger equations which are formal finite Taylor expansions of the linear pseudo-relativistic Schr\"odinger equation. In this paper, we establish global-in-time Strichartz estimates for these higher-order equations which hold uniformly in the speed of light. As nonlinear applications, we show that the higher-order Hartree(-Fock) equations approximate the corresponding pseudo-relativistic equation on an arbitrarily long time interval, with higher accuracy than the non-relativistic model. We also prove small data scattering for the higher-order nonlinear Schr\"odinger equations.  
\end{abstract}

\section{Introduction}

\subsection{Background}

We consider the pseudo-relativistic linear Schr\"odinger equation 
\begin{equation}\tag{pLS}\label{pLS}
i\hbar\partial_t \psi=\mathcal{H}^{(c)}\psi,
\end{equation}
where $\psi=\psi(t,x):\mathbb{R}\times\mathbb{R}^d\to\mathbb{C}$ is the wave function and the operator
$$\boxed{\quad\mathcal{H}^{(c)}=\sqrt{m^2c^4-c^2\hbar^2\Delta}-mc^2\quad}$$
is the Fourier multiplier of symbol $\sqrt{m^2c^4+c^2\hbar^2|\xi|^2}-mc^2$. Here, $m>0$ represents the particle mass, $\hbar$ is the reduced Plank constant, and $c\gg1$ denotes the speed of light. The operator $\mathcal{H}^{(c)}$ is called \textit{pseudo-relativistic} or \textit{semi-relativistic} in that it behaves both relativistically and non-relativistically depending on frequencies. Indeed, in the non-relativistic regime $|\xi|\ll\tfrac{mc}{\hbar}$, the Taylor series expansion yields
\begin{equation}\label{series expansion} 
\sqrt{\hbar^2c^2|\xi|^2+m^2c^4}-mc^2
=mc^2\left(\sqrt{1+\tfrac{\hbar^2|\xi|^2}{m^2c^2}}-1\right)=\tfrac{\hbar^2|\xi|^2}{2m}
-\tfrac{\hbar^4|\xi|^4}{8m^3c^2}+\cdots\approx\tfrac{\hbar^2|\xi|^2}{2m}.
\end{equation}

The pseudo-relativistic operator and associated linear and nonlinear models arise in various physics literature. For instance, the mean-field dynamics of relativistic fermion particles is described by the pseudo-relativistic Hartree-Fock equation
\begin{equation}\tag{pHF}\label{pNLHF}
    i\hbar\partial_t \psi_k =\mathcal{H}^{(c)}\psi_k+H\psi_k-F_k(\psi_k),\quad k=1,2,..., N,
\end{equation}
where $\psi_k=\psi_k(t,x):\mathbb{R}\times\mathbb{R}^3\to\mathbb{C}$, 
$$H\psi_k=\sum_{\ell=1}^N\left(\frac{\kappa}{|x|}*|\psi_\ell|^2\right)\psi_k$$
is the Hartree nonlinearity, and
$$F_k(\psi_k)=\sum_{\ell=1,\ell\neq k}^N\left(\frac{\kappa}{|x|}*(\overline{\psi_\ell}\psi_k)\right)\psi_\ell$$
is the Fock exchange term. A real constant $\kappa$ determines the strength of self-interaction among quantum particles; it is repulsive if $\kappa>0$, and attractive if $\kappa<0$. By the Pauli exclusion principle, $\psi_k$'s are assumed to be mutually $L^2$-orthogonal. Dropping the exchange term, the Hartree-Fock equation is reduced to the pseudo-relativistic Hartree equation 
\begin{align}\tag{pH}\label{pNLH}
    i\hbar\partial_t \psi_k =\mathcal{H}^{(c)}\psi_k+H\psi_k,\quad k=1,2,..., N.
\end{align}
Since the exchange term has lower-order in mean-field approximation, this Hartree model is considered as a simplified Hartree-Fock model. Without the orthogonal condition, the equation with $\kappa<-1$ describes the mean field dynamics of boson stars, so it is called the boson star equation. For the rigorous derivation of the semi-relativistic models, we refer to \cite{Lieb1984,Lieb1987,Elgart2007}. For the dynamics of the system, the Hartree(-Fock) equation is globally well-posed in the energy space $H^\frac12(\R^3)$ \cite{FL-2007, Lenzmann2007}. It is locally well-posed below the energy space \cite{Cho2006, Herr2014}. For dynamical properties, we refer to \cite{Lenzmann2007,Lenzmann2009,Lenzmann2011}.\\

For the relativistic models, an important question is to prove their non-relativistic limits $c\to\infty$. Note that by the convergence \eqref{series expansion} in low frequencies, it is expected that the pseudo-relativistic operator $\mathcal{H}^{(c)}$ in the above models ((pLS), (pHF) and (pH)) can be replaced by the non-relativistic one $-\frac{\hbar^2}{2m}\Delta$ in the non-relativistic regime, and that the non-relativistic linear Schr\"odinger equation and the Hartree(-Fock) equation are derived respectively. Indeed, proving the non-relativistic limits justifies consistency of the relativistic modification. On the other hand, it shows that non-relativistic models are good approximations to the pseudo-relativistic models, which are computationally extremely expensive due to the presence of the non-local operator $\mathcal{H}^{(c)}$.

Similar 
non-relativistic limits can be formulated for other types of relativistic models, and there have been numerous results on this subject. 
For instance, in the work of Machihara, Nakanishi and Ozawa \cite{MNO2002,Machihara2003}, the nonlinear Schrödinger equations are derived from the Klein-Gordon and the Dirac equations as non-relativistic limits. In \cite{Bechouche1998,Bechouche2004,Bechouche2005}, Bechouche, Mauser and Selberg established the convergence from the Dirac-Maxwell (resp., Klein-Gordon-Maxwell) system to the Vlasov-Poisson (resp., Schr\"odinger-Poisson) system. For the non-relativistic limits of the Dirac-Maxwell,  Klein-Gordon-Maxwell and Klein-Gordon-Zakharov systems, we refer to the work of Masmoudi and Nakanishi \cite{Masmoudi2002,Masmoudi2003, Masmoudi2005}.\\

For better accuracy to the non-local pseudo-relativistic equations, Carles and Moulay \cite{CM2012} introduced higher-order linear Schr\"odinger equations of the form
\begin{equation}\tag{hLS}\label{hLS}
i\hbar\partial_t\psi= \mathcal{H}_J^{(c)}\psi,
\end{equation}
where the operator $\mathcal{H}_J^{(c)}$ is a formal finite Taylor expansion of the pseudo-relativistic operator, that is, 
$$\boxed{\quad\mathcal{H}_{J}^{(c)}=-\sum_{j=1}^J\frac{\alpha(j)\hbar^{2j}}{m^{2j-1}c^{2j-2}}\Delta^{j}\quad}$$
with 
$$\alpha(j)=\frac{(2j-2)!}{j!(j-1)!2^{{2j-1}}}\quad (j\geq 1).$$
Note that such higher-order models include the non-relativistic Schr\"odinger equation $i\hbar\partial_t\psi=-\tfrac{\hbar^2}{2m}\Delta \psi$ and the fourth-order equation $i\hbar\partial_t \psi=(-\tfrac{\hbar^2}{2m}\Delta-\tfrac{\hbar^4}{8m^3c^2}\Delta^2\big)\psi$. In \cite{CLM2015}, Carles, Lucha and Moulay showed that the higher-order linear flow provides a more accurate approximation as $c\to\infty$ \cite[Theorem A.1]{CLM2015}, precisely, 
\begin{equation}\label{CLM linear approximation}
\|e^{it\mathcal{H}^{(c)}}\psi_0-e^{it\mathcal{H}_J^{(c)}}\psi_0\|_{L^2(\mathbb{R}^d)}\leq\frac{2T}{\hbar}\frac{\alpha(J+1)}{m^{2J+1}c^{2J}}\|\psi_0\|_{H^{2J+2}(\mathbb{R}^d)}.
\end{equation}
Moreover, employing the local-in-time Strichartz estimates \cite[Lemma 4.3]{CLM2015}, the authors developed a well-posedness theory for higher-order Hartree(-Fock) equations \cite[Theorem 4.9]{CLM2015}.

\subsection{Strichartz estimates}
In this paper, we establish ``global-in-time" Strichartz estimates for higher-order equations (hLS) which hold uniformly in the speed of light $c\geq1$. It turns out that such estimates are distinguished by odd and even orders of the Taylor expansion. For the statement, we call $(q,r)$ \textit{odd-admissible} if 
\begin{align}\label{o admissible pair}
2\le q,r\le \infty,\quad (q,r,d)\neq(2,\infty,2), \quad \frac 2q+\frac dr=\frac d2.
\end{align}
Our first main result asserts that for odd expansions, the uniform global-in-time Strichartz estimates hold for odd-admissible pairs.

\begin{theorem}[Strichartz estimates for \eqref{hLS}: odd case]
\label{Thm:Strichartz odd order HSE}
Let  $J\in 2\N-1$. Then, there exists a constant $A$, independent of $c\geq 1$, such that for any odd-admissible pairs $(q,r)$ and $(\tilde{q},\tilde{r})$, we have
\begin{align}\label{Strichartz for odd}
\big\| e^{-it \mathcal{H}_J^{(c)}}\psi_0\big\|_{L_t^q(\mathbb{R};L_x^r(\R^d))}
\le A \left( \frac{m}{\hbar}\right)^\frac1q\| \psi_0\|_{L^2(\R^d)}
\end{align}
and
\begin{align}\label{Inhomogeneous Strichartz for odd}
\Big\| \int_0^t e^{-i(t-s) \mathcal{H}_J^{(c)}}F(s)ds \Big\|_{L_t^q(\mathbb{R};L_x^r(\R^d))}\le
A^2  \left( \frac{m}{\hbar}\right)^\frac2q
\|F\|_{L_t^{\tilde{q}'}(\mathbb{R};L_x^{\tilde{r}'}(\R^d))}.
\end{align}
\end{theorem}

\begin{remark}
\begin{enumerate}
\item The odd-admissible pairs are exactly the same as those for the non-relativistic model. Indeed, the non-relativistic case $J=1$ is included.
\item The estimates \eqref{Strichartz for odd} and \eqref{Inhomogeneous Strichartz for odd} are independent of $c\geq1$, and they hold globally in time. In this sense, it improves the previous local-in-time bound \cite[Lemma 4.3]{CLM2015}. Moreover, they can be a proper analysis tool to analyze both the non-relativistic limit ($c\to\infty $) and the large-time asymptotics ($t\to\infty$) for nonlinear problems (see Theorem \ref{Thm:Approximation} and \ref{Thm:scattering}).
\end{enumerate}
\end{remark}

\begin{remark}
\begin{enumerate}
\item 
By simple but careful analysis on the symbol of the higher-order operator $\mathcal{H}_J^{(c)}$, we show that if $J\in\N$ is odd, the Hessian of the symbol has a strict lower bound, independent of $c\geq1$ (see Proposition \ref{odd hessian}). Once it is proved, Strichartz estimates follow immediately from the standard oscillatory integral theory (see Lemma \ref{Lem:Stationary phase method}).
\item In general, the symbols of higher-order Schr\"odinger operators may have degenerate Hessian. Such examples include the bi-harmonic operator $\Delta^2$ and the even-order expansion $H_{2J_0}^{(c)}$. The associate linear Schr\"odinger equations necessarily obey only weaker dispersive estimates. 
\item The proof of Theorem \ref{Thm:Strichartz odd order HSE} heavily relies on the algebraic structure of symbols, originated from the pseudo-relativistic operator. Indeed, the signs of ordered terms in the Hessian are alternating. Nevertheless, fortunately in odd case, negative terms can be controlled by neighboring positive terms. These are not true in general.
\end{enumerate}
\end{remark}

Next, we consider higher-order equations (hLS) with even Taylor expansions. In this case, the situation becomes much more complicated, because the Hessian of the symbol can be degenerate. Precisely, there are two spheres in the frequency space where the Hessian is degenerate (see Lemma \ref{Lem: even order rank}). Our second main result provides weaker Strichartz estimates for even expansions. We now call $(q,r)$ \textit{even-admissible} if 
\begin{equation}\label{e admissible pair}
2\le q\le \infty, \quad 2\le r< \infty,\quad
\left\{\begin{aligned}
&\frac {3}{q}+\frac {1}{r}=\frac 12 &&\text{ for } d=1,\\
&\frac {2}{q}+\frac {1}{r}=\frac 12 &&\text{ for } d\ge2.
\end{aligned}\right.
\end{equation}

\begin{theorem}[Strichartz estimates for \eqref{hLS}: even case]
\label{Thm:Strichartz even hLS}
Let $J\in 2N$. Then, there exists $A>0$, independent of $c\geq 1$, such that for an even-admissible pair $(q,r)$, 
\begin{align}\label{Strichartz for even}
\big\| e^{-it \mathcal{H}_J^{(c)}}\psi_0\big\|_{L_t^q(\mathbb{R};L_x^r(\R^d))}\le 
\left\{\begin{aligned}
&A\left( \frac{m}{\hbar}\right)^\frac1q\|\psi_0\|_{\dot H^{\frac1q}(\R^d)}&&\text{ for } d=1,\\
&A\left( \frac{m}{\hbar}\right)^\frac{1}{q} \|\psi_0\|_{\dot H^{\frac{2(d-1)}{q}}(\R^d)}&&\text{ for } d\ge2.
\end{aligned}\right.
\end{align}
\end{theorem}

\begin{remark}
\begin{enumerate}
\item In even case, we have weaker Strichartz estimates. Note that in odd case, the Strichartz estimates \eqref{Strichartz for odd} and the Sobolev inequality yield the estimates of the form \eqref{Strichartz for even}.
\item As mentioned above, in even case, the Hessian of the symbol has two degenerate spheres. It causes additional technical difficulties. To handle this, we need to employ the Littlewood-Paley decomposition.
\item We do not claim optimality of the estimates \eqref{Strichartz for even}. The additional derivative on the right hand side of \eqref{Strichartz for even} could be reduced, for instance, by applying a more delicate theory of the resolution of singularity \cite{Arnold1985}.
\end{enumerate}
\end{remark}

\subsection{Applications to nonlinear problems}

Strichartz estimates are one of the fundamental tools for linear and nonlinear dispersive PDEs (see \cite{Cazenavebook}). In this article, we discuss two nonlinear problems which can be solved by taking the full advantages of uniform global-in-time Strichartz estimates (Theorem \ref{Thm:Strichartz odd order HSE});  one is higher-order approximation estimate for the pseudo-relativistic Hartree(-Fock) equation via its higher-order expansions. The other is the small data scattering for the higher-order nonlinear Schr\"odinger equations.

\subsubsection{Higher-order approximation}
Let $J\in 2\N-1$ be an odd number. 
Then, as formal finite expansions of the pseudo-relativistic Hartree(-Fock)  model mentioned earlier, we introduce the corresponding higher-order Hartree-Fock equation
\begin{equation}\tag{hHF}\label{hNLS}
    i\hbar\partial_t \phi_k =\mathcal{H}_J^{(c)}\phi_k+H\phi_k - F_k(\phi_k),\quad k=1,2,..., N,
\end{equation}
and the higher-order Hartree equation
\begin{align}\tag{hH}\label{hS}
    i\hbar\partial_t \phi_k =\mathcal{H}_J^{(c)}\phi_k+H\phi_k,\quad k=1,2,..., N.
\end{align}
By Strichartz estimates (Theorem \ref{Thm:Strichartz odd order HSE}) and the mass conservation law, both (hHF) and (hH) are globally well-posed in $L^2(\mathbb{R}^3;\mathbb{C}^N)$ (see Proposition \ref{Pro:GlobalL2}). 

The following theorem provides precise global-in-time error bounds for the higher-order approximations of the pseudo-relativistic models.

\begin{theorem}[Higher-order approximation]\label{Thm:Approximation}
Let $J\in 2\N-1$ and $c\geq1$. Suppose that $\Psi_0=\{\psi_{k,0}\}_{k=1}^N\in  H^\frac12(\R^3;\mathbb{C}^N)$. If $\kappa<0$, we further assume that $\|\Psi_0\|_{H^\frac12(\mathbb{R}^3;\mathbb{C}^N)}$ is sufficiently small. Let $\Psi^{(c)}(t)=\{\psi_k^{(c)}(t)\}_{k=1}^N \in C(\R;H^\frac12(\R^3;\mathbb{C}^N))$ be the global solution to \eqref{pNLHF} (resp., \eqref{pNLH}) with initial data $\Psi_0$, and let $\Phi^{(c)}(t)=\{\phi_k^{(c)}(t)\}_{k=1}^N \in C(\R;H^\frac12(\R^3;\mathbb{C}^N))$ be the global solution to \eqref{hNLS} (resp., \eqref{hS}) with the same initial data. Then, there exist $A,B>0$, depending on $\|\Psi_0\|_{H^\frac12(\mathbb{R}^3;\mathbb{C}^N)}$ but  independent of $c\geq1$, such that 
\begin{align}\label{approximation}
\| \Phi^{(c)}(t)- \Psi^{(c)}(t)\|_{ L^2(\R^3;\mathbb{C}^N) }\le Ac^{-\frac{J}{2(J+1)}}e^{Bt}.
\end{align}
\end{theorem}

\begin{remark}
\begin{enumerate}
\item The limit behavior of stationary states to the pseudo-relativistic equations have been studied. In \cite{Lenzmann2009, Choi2016, CHS2018-0}, the authors established the non-relativistic limits of stationary states, which corresponds to the $J=1$ case. Then, the higher-order approximation estimates are proved \cite{CHS2018}.
\item   The convergence rate in \eqref{approximation} is getting better as $J$ increases, but we do not claim its optimality. 
\item The proof is quite standard. We write the solutions in Duhamel formulae and measure the difference in $L^2$. The convergence rate comes from the Taylor expansion of symbol and the regularity gap from $H^\frac12$ where the initial data is given. 
\item The smallness condition on initial data when $\lambda<0$ is imposed just for the global well-posedness of \eqref{pNLHF}. Even $J$'s are not included, but the same result can be proved once the loss of regularity in Strichartz estimates is reduced as for odd $J$'s. The argument can be easily applied to higher dimensional case $d\ge3$. 
\end{enumerate}
\end{remark}

\subsubsection{Small data scattering}
Let $J\in 2\N-1$. We now consider the higher-order nonlinear Schr\"odinger equation, 
\begin{equation}\label{gNLS}\tag{hNLS}
  i\partial_t\psi= \mathcal{H}_J^{(c)}\psi
  +\kappa|\psi|^{\nu-1}\psi,
  \end{equation}
  where $\psi=\psi(t,x):\mathbb{R}\times\mathbb{R}^d\to\mathbb{C}$, $\kappa\in\R$ and $\nu>1$. This equation could be considered as the $J$ th-order approximation to the pseudo-relativistic NLS
 $ i\partial_t \psi=\mathcal{H}^{(c)}\psi+\kappa|\psi|^{\nu-1}\psi$. Using the Strichartz estimates (Theorem \ref{Thm:Strichartz odd order HSE}), we prove small data scattering.
 
\begin{theorem}[Small data scattering in $H^1(\R^d)$ for \eqref{gNLS}]\label{Thm:scattering}
  Let $J\in 2\N-1$. Suppose that $\nu$ satisfies 
  \begin{align}\label{H1 LWP}
  \begin{cases}
   1+\frac{4}{d}<\nu<\frac{d+2}{d-2}, & \text{ if } d\ge 3 \\
   1+\frac{4}{d}<\nu<\infty, & \text{ if } d=1,2.
  \end{cases}
  \end{align}
  Then, there exists $\ep_0>0$ (independent of $c\geq 1$) such that if $\|\psi_{c,0}\|_{H^1(\mathbf{R}^d)}\le \ep_0$, then there exists a unique global solution $\psi_c(t)\in C(\R;H^1(\R^d))$ to \eqref{gNLS} with initial data $\psi_{c,0}$, and it moreover scatters in $H^1(\R^d)$, i.e., there exist $\psi_{J,\pm}^{(c)}\in H^1(\R^d)$ such that
  \begin{align*}
  \lim_{t\rightarrow \pm\infty} \| \psi_c(t) - e^{-it\mathcal{H}_J^{(c)}}\psi_{J,\pm}^{(c)}\|_{H^1(\mathbb{R}^d)}=0.
  \end{align*}
  \end{theorem}
  
\begin{remark}
\begin{enumerate}
\item The proof of \eqref{Thm:scattering} is identical to that in the non-relativistic case (see \cite{Cazenavebook}), because Strichartz estimates \eqref{Strichartz for odd} are exactly the same. 
\item An interesting question would be to show how accurately the scattering state $\psi_{J,\pm}^{(c)}$ for the higher-order model approximates the relativistic scattering state. The non-relativistic limit $(J=1)$ of scattering states is proved for the nonlinear Klein-Gordon equation \cite{Nakanishi2001}. 
\end{enumerate}
\end{remark}

\subsection{Organization}
In Section~2, we provide basic lemmas for oscillatory integrals, which will be repeatedly used to estimate the kernel of linear flows.
In Section~3, we prove the global Strichartz estimates which are uniform in $c$ when $J$ is odd.
In Section~4, the Strichartz estimates with derivative loss are established when $J$ is even, which requires additional technical issues.
In Section~5, we arrange well-posedness results for \eqref{pNLHF} and prove the global well-posedness of \eqref{hNLS} as applications of uniform Strichartz estimates. In addition, we show that the nonlinear solutions are uniformly bounded with respect to $c$. Finally, we prove the convergence of solutions to \eqref{hNLS} towards solutions to \eqref{pNLHF} as $c$ goes to infinity.
In Section~6, as an application of global Strichartz estimates, we consider the power type nonlinearity and prove the small data scattering result.

\subsection*{Acknowledgement}
This research of the first author was supported by the Basic Science Research Program through the National Research Foundation of Korea (NRF) funded by the Ministry of Science and ICT (NRF-2020R1A2C4002615).
C. K. was supported by the National Research Foundation of Korea (NRF) grant funded by the Korea government (MSIT) (No. 2020R1F1A1A0106876811). C. Yang was supported by the National Research Foundation of Korea(NRF) grant funded by the Korea government(MSIT) (No. 2021R1C1C1005700).

The authors are grateful to the anonymous referee for careful reading of the paper and valuable comments.

\section{Preliminaries}

\subsection{Notations}
Let $\xi=(\xi_1,\cdots,\xi_d)\in\R^d$.
\begin{itemize}
  \item $\{ \mathbf{e}_j \}_{j=1}^d$ denotes the standard basis of $\R^d$.
  \item We identify a vector in $\R^d$ as a $d \times 1$ column matrix via the standard isomorphism 
  \begin{align*}
    (\xi_1,\cdots,\xi_d) \ \longleftrightarrow \
    \begin{bmatrix}
      \xi_1 \\ 
      \vdots \\ 
      \xi_d
    \end{bmatrix}
  \end{align*}
  \item 
  For $\eta \in \R^d$, we denote the inner product by
  $\xi\cdot\eta=\sum_{i=1}^d\xi_i\eta_i$.
  \item Fix $j \in\{1,\cdots,d\}$. We denote
  $\check{\xi}_j=(\xi_1,\cdots,\xi_{j-1},\xi_{j+1},\cdots,\xi_d)\in\R^{d-1}$.
  \item For an $n$-tuple $\alpha=(\alpha_1,\cdots,\alpha_d)$ of nonnegative integers and a function $f$ on $\R^d$, we define 
  \begin{align*}
     \partial^\alpha f := \partial_1^{\alpha_1}\partial_2^{\alpha_2} \cdots \partial_d^{\alpha_d} f
  \end{align*}
\end{itemize}

Let $\Omega:\R^d\rightarrow\R$ be a $C^2$ function.
We denote the Hessian matrix of $\Omega$, $d\times d$ matrix, by $\textbf{H}(\Omega)$ whose $(i,j)$ component is given by 
$$[\textbf{H}(\Omega)]_{ij} = \partial_{\mathbf{e}_i}\partial_{\mathbf{e}_j}\Omega \text{ for } i,j=1,2,\cdots,d.$$
We say that $\Omega$ is \textit{degenerate} at $x\in \R^d$ if det$\textbf{H}(\Omega)(x)=0$.

We generalize the notion of Hessian.
Let $E$ be a set of $k$ orthonormal vectors in $\R^d$, say, 
$E=\{ \mathbf{u}_{1}, \cdots, \mathbf{u}_{k} \}$.
We define \textit{a Hessian of $\Omega$ with respect to $E$} by $k\times k$ matrix $\textbf{H}_{\{ \mathbf{u}_{1}, \cdots, \mathbf{u}_{k} \}}(\Omega)$ whose $(m,n)$ component is given by 
$$[\textbf{H}_{\{ \mathbf{u}_{1}, \cdots, \mathbf{u}_{k} \}}(\Omega)]_{mn} = \partial_{\mathbf{u}_{m}}\partial_{\mathbf{u}_{n}}\Omega \text{ for } m,n=1,2,\cdots,k,$$ where $\partial_{\mathbf{u}_{m}}$ denotes the directional derivative along $\mathbf{u}_{m}$.


The rank of a matrix $A$ is the dimension of the vector space generated by its columns. If $A$ has a rank $k$, we denote rank$A=k$.

Let $\chi\in C_c^\infty(\mathbb{R})$ be a radial non increasing function such that
$\chi(x)=1$ for $|x|\le 1$ and $\chi(x)=0$ for $|x|\ge2$, and
$$\sum_{N\in \Z}\chi_N(|x|) \equiv 1, \text{ for } x\in \R^d\setminus\{0\},$$
where $\chi_N=\chi(\frac{\cdot}{2^{N+1}})-\chi(\frac{\cdot}{2^N})$ for $N\in \Z$.
We define the projection operator $P_N$ by the fourier multiplier such that
\begin{align}\label{projection operator}
\widehat{P_Nf}(\xi)=\chi_N(|\xi|)\widehat{f}(\xi).
\end{align}
Then, $f= \sum_{N\in \Z}P_N f$.

Let us define $\Theta_j=\{ \xi\in\R^d : |\xi_i|\ge \frac{1}{\sqrt{2d}}|\xi| \}$ with $j=1,2,\cdots,d$. Then, $\R^d\setminus\{0\}= \cup_{j=1}^d \Theta_i$ and there is a partition of unity $\{\theta^j\}$ subordinate to the covering $\Theta_j$ (the $\theta^j$ can be defined on the sphere and extended such that they are homogeneous of order zero) satisfying that 
\begin{align} \label{partition of unity}
  \sum_{j=1}^d \theta^j(\xi) = 1 \text{ and }
  |\xi_j|\ge \frac{1}{\sqrt{2d}}|\xi|  \text{ on the support of } \theta^j .
 \end{align}

\subsection{Stationary phase method.}

By taking the Fourier transform, the linear solution to \eqref{hLS} is given by
\begin{align}\begin{aligned} \label{oscillatory integral}
  e^{-it\mathcal{H}_J^{(c)}}\psi_0(x) 
  &= \frac{1}{(2\pi)^d}\int_{\R^d}  
  e^{ix\cdot\xi-it\Omega_J^{(c)}(\xi)} \widehat{\psi_0}(\xi) d\xi,
\end{aligned}\end{align}
where the dispersion relation is given by 
$$\Omega_J^{(c)}(\xi):=\sum_{j=1}^J\frac{(-1)^{j+1}(2j-2)!\hbar^{2j-1}}{(j-1)!j!(2m)^{2j-1}c^{2j-2}}|\xi|^{2j}.$$
We observe that $\Omega_J^{(c)}$ is a radial function, so we have 
$$\Omega_J^{(c)}(\xi)=\omega_J^{(c)}(|\xi|) \text{ with } \omega_J^{(c)}:[0,\infty)\rightarrow \R,$$
explicitly, 
\begin{align}\label{def omega}
  \omega_J^{(c)}(r)=\sum_{j=1}^J\frac{(-1)^{j+1}(2j-2)!\hbar^{2j-1}}{(j-1)!j!(2m)^{2j-1}c^{2j-2}}r^{2j}.
\end{align}
The linear solution can be written as kernel form 
\begin{align*}
  e^{-it\mathcal{H}_J^{(c)}}\psi_0(x) 
  = \frac{1}{(2\pi)^d}\int_{\R^d}  \left( \int_{\R^d}
  e^{i(x-y)\cdot\xi-it\Omega_J^{(c)}(\xi)} d\xi \right) \psi_0(y) d\xi
  =\mathcal{I}_{J}^{(c)}\left(t,\tfrac{\cdot}{t}\right)\ast \psi_0(x),
\end{align*}
where we introduced the oscillatory integral
\begin{align*}
  \mathcal{I}_{J}^{(c)}(t,v)  =\frac{1}{(2\pi)^d}\int_{\R^d} 
  e^{it(v\cdot\xi-\Omega_J^{(c)}(\xi))}d \xi.
\end{align*}
In this subsection, we list basic analysis tools to estimate the oscillatory integrals. 
For one dimensional case, the decay rate of oscillatory integral is totally determined from the nondegeneracy of the derivatives of the phase function (see \cite[Chapter~VIII]{Steinbook-Hrmonic analysis}).

\begin{lemma}[Van der Corput Lemma]\label{Lem:Van}
  Let $k\in\Z^+$ and $|\phi^{(k)}(x)|\ge1$ for all $x\in[a,b]$ with $\phi'(x)$ monotonic in the case $k=1$. Then,
  \begin{align}\label{1d decay}
    \left| \int_{a}^b e^{i\lambda\phi(x)} dx \right| \le C_k\lambda^{-\frac{1}{k}}
  \end{align}
  and 
  \begin{align*}
  \left|\int_{a}^b e^{i\lambda\phi(x)} \eta(x)dx\right| \le C_k\left(\int_a^b|\eta'(x)|dx+|\eta(b)|\right) \lambda^{-\frac{1}{k}},
  \end{align*}
  where the constant $C_k$ is independent of $a,b$ and $\phi$.
  \end{lemma}

  For multi dimensional case, the rank of Hessian matrix of the phase function plays a crucial role (see \cite[Chapter~8]{Steinbook-Functionalanalysis}).
  \begin{lemma}[Stationary phase method]\label{Lem:Stationary phase method}
  For given a real-valued phase function $\Omega\in C^\infty(\R^d)$ and amplitude function $\eta\in C_0^\infty(\R^d)$, assume that 
  \begin{align}
   \textup{rank}\textbf{\textup{H}}(\Omega) = m \text{ on the support of } \eta,
  \end{align}
  for $0<m\le d$, in other words,
   there exist a coordinate $x=(x',x'')\in \R^m\times\R^{d-m}$ 
   with a basis $\{\mathbf{u}_j\}_{j=1}^d$
   such that 
  \begin{align}
    \left| \textup{det} \He_{\{\mathbf{u}_1,\cdots,\mathbf{u}_m\}}\Omega \right| \ge 1 \text{ on the support of } \eta.
   \end{align}
  Then,
  \begin{align*}
  \left| \int_{\R^d}e^{i\lambda \Omega(x)}\eta(x)dx\right|\le 
  C_m \left|\textup{supp}\;\eta\right|
  C_{\Omega,\eta}\lambda^{-\frac m2},
  \end{align*}
  where 
  \begin{align*}
  (C_{\Omega,\eta})^2 &\le  1+\|\eta\|_{C^{m+1}} \max_{2\le |\alpha| \le m+2}\{
    C_\alpha : \sup_{x\in \textup{supp} \eta} |\partial^\alpha\Omega(x)|\le C_\alpha, \},\end{align*}
  and $C_m$ depends only on $m$.
  \end{lemma}

\subsection{Hessian of radial function}
In the previous subsection, we found that the Hessian matrix of phase function played an essential role in analysis of oscillatory integrals.
Recall that $\Omega_J^{(c)}$ in the phase function in \eqref{oscillatory integral} is radial, so the formula for Hessian is quite simple.
\begin{lemma}[Hessian of radial function]\label{Lem:Hessian}
Let $\omega:[0,\infty)\rightarrow\R$ and $\Omega:\R^d\rightarrow \R$ be given by $\Omega(\xi)=\omega(|\xi|)$.
Then, for $\xi\in\R^d\setminus\{0\}$
\begin{equation}\label{radial Hessian}
\textup{det}(\He\Omega)(\xi)
=\omega''(|\xi|)\big\{\omega'(|\xi|)|\xi|^{-1}\big\}^{d-1}.
\end{equation}
\end{lemma}
\begin{proof}
Without loss of generality, we assume $\xi_1\neq0$.
We compute 
\begin{align*}
  \partial_{\mathbf{e}_i}\partial_{\mathbf{e}_j}\Omega(\xi)
&=\omega''(|\xi|)\frac{\xi_i\xi_j}{|\xi|^2}+\omega'(|\xi|)\Big( \frac{\delta_{ij}}{|\xi|}-\frac{\xi_i\xi_j}{|\xi|^3} \Big) \\
&=\frac{\omega'(|\xi|)}{|\xi|}\delta_{ij}+\Big( \omega''(|\xi|)-\frac{\omega'(|\xi|)}{|\xi|} \Big)\frac{\xi_j\xi_j}{|\xi|^2}.
\end{align*}
For notational convenience, we denote 
\begin{equation}\label{A and B}
A=\frac{\omega'(|\xi|)}{|\xi|},\quad B= \omega''(|\xi|)-\frac{\omega'(|\xi|)}{|\xi|}.
\end{equation}
Then, we have $\nabla\partial_{\mathbf{e}_j}\Omega(\xi)=A\mathbf{e}_j+(B\frac{\xi_j}{|\xi|^2})\xi$. Thus, by Gaussian elimination (or the row reduction), we obtain
\begin{align}\begin{aligned}\label{GE}
  \textup{det}(\He\Omega)(\xi)&=\textup{det}\begin{bmatrix}
A\mathbf{e}_1+(B\frac{\xi_1}{|\xi|^2})\xi\\
A\mathbf{e}_2+(B\frac{\xi_2}{|\xi|^2})\xi\\
\cdots\\
A\mathbf{e}_d+(B\frac{\xi_d}{|\xi|^2})\xi
\end{bmatrix}=\textup{det}\begin{bmatrix}
A\mathbf{e}_1+(B\frac{\xi_1}{|\xi|^2})\xi\\
A\mathbf{e}_2-\frac{\xi_2}{\xi_1}A\mathbf{e}_1\\
\cdots\\
A\mathbf{e}_d-\frac{\xi_d}{\xi_1}A\mathbf{e}_1
\end{bmatrix}\\
&=\textup{det}\begin{bmatrix}
(A+B)\mathbf{e}_1\\
A\mathbf{e}_2-\frac{\xi_2}{\xi_1}A\mathbf{e}_1\\
\cdots\\
A\mathbf{e}_d-\frac{\xi_d}{\xi_1}Ae_1
\end{bmatrix}=\textup{det}\begin{bmatrix}
(A+B)\mathbf{e}_1\\
A\mathbf{e}_2\\
\cdots\\
A\mathbf{e}_d
\end{bmatrix}=A^{d-1}(A+B).
\end{aligned}\end{align}
Since $A+B=\omega''(|\xi|)$, we prove the desired formula.
\end{proof}

\section{Strichartz estimates: Odd case}

Throughout of the paper, we are interested in asymptotic phenomena of higher-order equation in terms of $c$, thus, in what follows, we fix $\hbar=m=1$ via rescaling, for simplicity of calculation.

When $J=1$ corresponding to the Schrödinger equations,
we can easily show from \eqref{radial Hessian} that 
$$\textup{det}\He \left( v\cdot\xi-\tfrac12|\xi|^2\right)=1 \text{ for all } \xi\in\R^d, $$ which gives by Lemma~\ref{Lem:Stationary phase method} the dispersive estimates
\begin{align*} 
 \sup_{v\in\R^d}\frac{1}{(2\pi)^d}\int_{\R^d} 
 e^{it(v\cdot\xi-\frac12|\xi|^2)}d \xi \le A t^{-\frac d2}.
\end{align*}
In this section, we show that the argument 
can be extended to all odd cases $J\ge3$.

\begin{proposition}\label{odd hessian}
  Let $J\ge3$ be an odd integer. Then, 
  \begin{align} \label{det Hessian}
    \textup{det}\He (\Omega^{(c)}_J)(\xi)
    \gs (1+\tfrac{|\xi|}c)^{(2J-2)d},    
    \end{align}
and as a consequence we have
  \begin{align}\label{time decay}
    \sup_{v\in\R^d} \left| \mathcal{I}_J^{(c)}(t,v) \right| \ls  t^{-\frac d2},
  \end{align}  
  where the implicit constants are independent of $c$.
\end{proposition}
\begin{proof}
First, we show that \eqref{time decay} follows from  
\eqref{det Hessian}. When $d=1$, it is direct application of 
Lemma~\ref{Lem:Van}.
Also, when $d\ge2$, it can be shown by applying Lemma~\ref{Lem:Stationary phase method}. Indeed, we write 
\begin{align*}
  \mathcal{I}_{J}^{(c)}(t,v) =\frac{1}{(2\pi)^d}\int_{\R^d} 
  e^{it(v\cdot\xi-\Omega_J^{(c)}(\xi))} \chi(\tfrac{\xi}{c})d \xi
  +\sum_{\substack{N \in \Z \\ 2^N> c}} 
  \frac{1}{(2\pi)^d}\int_{\R^d} 
  e^{it(v\cdot\xi-\Omega_J^{(c)}(\xi))} \chi_N(\xi)d \xi,
\end{align*}
and change variables to obtain 
\[\begin{aligned}
 \mathcal{I}_{J}^{(c)}(t,v) &= 
 \tfrac{c^d }{(2\pi)^d}
 \int_{\R^d} e^{ic^2t(c^{-1}v\cdot\xi-c^{-2}\Omega_J^{(c)}(c\xi))} \chi(\xi)d \xi \\ 
 &\quad +\sum_{\substack{N \in \Z \\ 2^N> c}} \tfrac{2^{Nd}}{(2\pi)^d}  \int_{\R^d} e^{i2^{2N}t(2^{-N}v\cdot\xi-2^{-2N}\Omega_J^{(c)}(2^N\xi))} \chi_0(\xi)d \xi\\
 &=: \mathcal O_1 + \mathcal O_2.
\end{aligned}\]

Once proving that for $R>0$
\begin{equation}\label{eq:Oscillatory}
  R^d\left|  \int_{\R^d} e^{iR^2t(R^{-1}v\cdot\xi-R^{-2}\Omega_J^{(c)}(R\xi))} \eta(\xi)d \xi \right| 
  \ls 
  \begin{cases}
    \left(1+\tfrac{R}c\right)^{(J-1)}t^{-\frac d2} &\text{ for } \eta=\chi, \\ 
    \left(1+\tfrac{R}c\right)^{(J-1)(1-d)}t^{-\frac d2} &\text{ for } \eta=\chi_0, 
  \end{cases}
\end{equation}
we immediately have 
\[\mathcal O_1 \lesssim t^{-\frac{d}{2}},\]
and
\[\mathcal O_2 \lesssim \sum_{\substack{N \in \Z \\ 2^N > c}}  \left(1+\tfrac{2^N}c\right)^{(J-1)(1-d)}t^{-\frac d2} \lesssim t^{-\frac d2}\] 
due to $J\ge 3$ and $d\ge2$.  From \eqref{det Hessian}, we see that 
\begin{equation}\label{eq:prop3.1-1}
  \textup{det}\He \left( R^{-1}v\cdot\xi-R^{-2}\Omega_J^{(c)}(R\xi) \right)
 = \left[ \textup{det}\He (\Omega^{(c)}_J)\right](R\xi) \gs  \left(1+\tfrac{R|\xi|}{c}\right)^{(2J-2)d},
\end{equation}
and a direct computation gives 
\begin{equation}\label{eq:prop3.1-2}
\sup_{2\le|\alpha|\le d+2}  \Big\{
    C_\alpha : \sup_{\xi \in \textup{supp} \eta}\big|\partial_{\xi}^\alpha
  \big\{ (R^{-1}v\cdot\xi-R^{-2}\Omega_J^{(c)}(R\xi)) \big\} \big|\le C_\alpha, \Big\}
  \ls \left(1+\tfrac{R}c\right)^{2J-2}.
\end{equation}
Note that the right-hand side of \eqref{eq:prop3.1-1} is further bounded below by $1$ for $\eta = \chi$, and $\left(\tfrac{R}{c}\right)^{(2J-2)d}$ for $\eta = \chi_0$. Note also that the implicit constants in \eqref{eq:prop3.1-1} and \eqref{eq:prop3.1-2} are independent of both $v\in\R^d$ and $c$, but dependent on $J$. Now, we apply Lemma~\ref{Lem:Stationary phase method} to obtain 
\[  \mbox{LHS of } \eqref{eq:Oscillatory}
  \ls 
  \begin{cases}
    R^d\left(1+\tfrac{R}c\right)^{\frac12(2J-2)}\left( R^2t \right)^{-\frac d2} &\text{ for } \eta=\chi, \\ 
    R^d\left(1+\tfrac{R}c\right)^{\frac12(2J-2)}\left( \left(1+\tfrac{R}c\right)^{(2J-2)}R^2t \right)^{-\frac d2} &\text{ for } \eta=\chi_0, 
  \end{cases}
\]
which proves \eqref{eq:Oscillatory}.

Next, we prove  \eqref{det Hessian}. Recall from \eqref{radial Hessian} that 
\begin{align*}
  \textup{det}\He (\Omega_J^{(c)})(\xi)
  = (\omega_J^{(c)})''(|\xi|)\big\{ (\omega_J^{(c)})'(|\xi|)|\xi|^{-1}\big\}^{d-1}.
\end{align*}
Let $J=2J_0-1$ with $J_0\in \N$.
For $(\omega^{(c)}_J)'$, collecting the positive and negative terms, we write
$$r^{-1}  (\omega^{(c)}_J)'(r)=\sum_{j=0}^{J_0-1}\frac{(4j)!}{((2j)!)^22^{4j}c^{4j}}r^{4j}-\sum_{j=1}^{J_0-1}\frac{(4j-2)!}{((2j-1)!)^22^{4j-2}c^{4j-2}}r^{4j-2}.$$
For the second term, by the Cauchy-Schwarz inequality, we have
$$\begin{aligned}\frac{(4j-2)!}{((2j-1)!)^22^{4j-2}c^{4j-2}}r^{4j-2}
\leq\frac{1}{2}\frac{(4j-4)!}{((2j-2)!)^2 2^{4j-4}c^{4j-4}}r^{4j-4}+\frac{1}{2}\frac{(4j)!}{((2j)!)^22^{4j}c^{4j}}r^{4j},
\end{aligned}$$
which proves
$$r^{-1}  (\omega^{(c)}_J)'(r)\geq\frac{1}{2}+\frac{(2J-2)!}{((J-1)!)^22^{2J-2}}\left(\frac{r}{c}\right)^{2J-2}.$$
Similarly for $(\omega^{(c)}_J)''(r)$, we write
\begin{align*}
  (\omega^{(c)}_J)''(r)
&=\sum_{j=1}^{2J_0-1}\frac{(-1)^{j+1}(2j-1)!}{(j-1)!(j-1)!2^{2j-2}c^{2j-2}}r^{2j-2} \\
&=
\sum_{j=0}^{J_0-1}\frac{(4j+1)!}{((2j)!)^22^{4j}c^{4j}}r^{4j}
-\sum_{j=1}^{J_0-1}\frac{(4j-1)!}{((2j-1)!)^22^{4j-2}c^{4j-2}}r^{4j-2}\\
&=1
+\sum_{j=1}^{J_0-1}\frac{(4j+1)!}{((2j)!)^22^{4j}c^{4j}}r^{4j}\Big(1-\frac{c^2}{r^2}\frac{4j}{4j+1}\Big).
\end{align*}
When $r \geq c$, one can see that the summand is positive for all $1 \le j \le J_0-1$, thus we have
$$(\omega^{(c)}_J)''(r)\geq 1 + \frac{(2J-2)!}{((J-1)!)^22^{2J-2}c^{2J-2}}r^{2J-2}.$$
On the other hand, for $0\leq r<c$, since $\omega^{(c)}_J$ is Taylor expansion of
\begin{align*} 
c^2 \Big\{ \sqrt{ 1+\left(\tfrac{r}{c}\right)^2 }  -1 \Big\},
\end{align*}
by the term-by-term differentiation and Taylor's theorem, we obtain
%
\begin{align*} 
\left(1+\left(\frac{r}{c}\right)^2 \right)^{-\frac32}
=(\omega^{(c)}_J)''(r) +  \frac{(-1)^J(2J+1)!}{(J)!(J)!2^{2J}c^{2J}}(r_*)^{2J}  
\end{align*}
for some $r_*\in[0,r)$.
Since $J$ is odd, we conclude that
for $0\leq r<c$
\begin{align*}
  (\omega^{(c)}_J)''(r) 
 \ge \left(1+\left(\frac{r}{c}\right)^2 \right)^{-\frac32} 
 \ge 2^{-\frac32}.
\end{align*}
\end{proof}
\color{black}

\begin{remark}
  For one dimensional case, we applied Lemma~\ref{Lem:Van} in the above proof, where the constant in \eqref{1d decay} only depends on the lower bound of second derivative of phase function. Thus, one sees that if $d=1$, the implicit constant in \eqref{time decay} is also independent of $J$.
\end{remark}

Recall that 
\begin{align*}
  e^{-it\mathcal{H}_J^{(c)}}\psi_0(x) 
  =\mathcal{I}_{J}^{(c)}\left(t,\tfrac{\cdot}{t}\right)\ast \psi_0(x).
\end{align*}
So, by Young's inequality and \eqref{time decay} we have 
\begin{align*}
  \left\|   e^{-it\mathcal{H}_J^{(c)}}\psi_0 \right\|_{L^\infty(\R^d) }
  \le  \sup_{v\in\R^d} \left| \mathcal{I}_J^{(c)}(t,v) \right| \|\psi_0\|_{L^1}
  \le At^{-\frac d2} \| f\|_{L^1(\R^d)}.
\end{align*}
It is well-known that \eqref{hLS} enjoys mass conservation law
\begin{align*} 
\left \| e^{-it\mathcal{H}_J^{(c)}}\psi_0 \right\|_{ L^2(\R^d) }=\| \psi_0 \|_{ L^2(\R^d) }
\text{ for all } t\in \R.
\end{align*}
Interpolating two estimates, we obtain
for $2\le p \le \infty$
\begin{align*} 
  \left \| e^{-it\mathcal{H}_J^{(c)}}\psi_0 \right\|_{ L^p(\R^d) }
   \ls 
   t^{-\frac d2(1-\frac2p)}
   \| f \|_{ L^{p'}(\R^d) }.
\end{align*}
Now, the Strichartz estimates \eqref{Strichartz for odd} follows from the well-known $TT^*$ argument in \cite{KT1998}. We omit the details.


\section{Strichartz estimates: Even case}
As we did in the odd case, we will apply Lemma~\ref{Lem:Stationary phase method} to estimate the kernel of linear solution. So, we begin with describing the properties of the dispersion function, $\Omega^{(c)}_J$.
\subsection{Hessian of the phase function}
We examine the behavior of the radial function $\omega^{(c)}_J$ satisfying that  $\Omega^{(c)}_J(\xi)=\omega^{(c)}_J(|\xi|)$ for even case, $J\in 2\N$ (see \eqref{def omega}).      
\begin{lemma}\label{Lem:even order phase function}
  Let $J\in 2\N$. Then, there exists $A>0$ only depending on $J$ satisfying the followings:
  \begin{enumerate}
    \item 
    $(\omega^{(c)}_J)'''(r) \le -3\cdot2^{-\frac52}c^{-2}r$.  \label{bound for 3 omega}
  \item $(\omega^{(c)}_J)''$ has a unique zero $r_2$ in $(\frac c2,c)$. Furthermore,
$$ 
(\omega^{(c)}_J)''(r)>\tfrac58 \text{ for } r\le \tfrac c2 \text{ and } 
(\omega^{(c)}_J)''(r)<-A\left( \tfrac{r}{c}\right)^{2J-2} \text{ for } r\ge c.
$$
 \item $(\omega^{(c)}_J)'$ has a unique zero $r_1\in (c,2c)$. Furthermore, 
    $$r^{-1}(\omega^{(c)}_J)'(r)>\tfrac{1}{2} \text{ for }r<c \text{ and }      
    r^{-1}(\omega^{(c)}_J)'(r)< -A\left(\tfrac{r}{c} \right)^{2J-2} \text{ for }  r>2c.$$
  \end{enumerate} 
\end{lemma}

  \begin{proof} Let $J=2J_0$ with $J_0\in\N$.
  We begin with verifying estimates \eqref{bound for 3 omega} for the third derivative. For $r\ge c$, we have
  \begin{align*}
    (\omega^{(c)}_J)'''(r)&=\sum_{j=2}^{2J_0}\frac{(-1)^{j+1}(2j-1)!}{(j-2)!(j-1)!2^{2j-3}c^{2j-2}}r^{2j-3} \\
  &=-\frac{3r}{c^2}+ \sum_{j=1}^{J_0-1}\frac{(4j+1)!}{(2j)!(2j-1)!2^{4j-1}c^{4j}}r^{4j-1}
  -\sum_{j=1}^{J_0-1}\frac{(4j+3)!}{(2j+1)!(2j)!2^{4j+1}c^{4j+2}}r^{4j+1} \\
  &=-\frac{3r}{c^2}+\sum_{j=1}^{J_0-1}\frac{(4j+1)!}{(2j)!(2j-1)!2^{4j-1}c^{4j+2} }r^{4j-1}
  \Big( c^2 - \frac{(4j+3)}{4j}r^2 \Big) \le -\frac{3r}{c}
  \end{align*}
  On the interval $(0,c)$, $(\omega^{(c)}_J)'''$ is the $(J-2)$-th degree Taylor polynomial of $-3c^{-1}\frac{r}{c}(1+(\frac{r}{c})^2)^{-\frac52}$. Hence, for $0<r<c$, there exists $0<r_*<r$ such that
  \begin{align}\label{Taylor omega3 less than c}
    -3c^{-1}\frac{r}{c}\left(1+\left(\frac{r}{c}\right)^2\right)^{-\frac52}
  =(\omega^{(c)}_J)'''(r)
  +\frac{(4J_0+1)!}{(2J_0)!(2J_0-1)!2^{4J_0-1}c^{4J_0}}r_*^{4J_0-1},
  \end{align}
  which implies that 
  \begin{align*}
    (\omega^{(c)}_J)'''(r)
  \le  -3c^{-1}\frac{r}{c}\left(1+\left(\frac{r}{c}\right)^2\right)^{-\frac52}
  \le -3\cdot2^{-\frac52}\frac{r}{c^2}.
  \end{align*}

Secondly, a direct computation gives that 
  \begin{align*}
    (\omega^{(c)}_J)''(r)
    &=\sum_{j=1}^{J_0}\frac{(4j-3)!r^{4j-2}}{(2j-2)!(2j-2)!2^{4j-4}c^{4j-2}}
    \Big( \frac{c^2}{r^2}-\frac{4j-1}{4j-2}\Big). 
    \end{align*}
Since $(\omega^{(c)}_J)''$ is decreasing on $[0,\infty)$, we have 
$$  (\omega^{(c)}_J)''(r)>(\omega^{(c)}_J)''(c/2)> \tfrac58,  \ \text{ for } \ r<\tfrac c2.
$$
Also, since $ 1<\frac{4j-1}{4j-2}\le \frac32$ for all $j\in\N$, all terms in the summation are negative for $r>c$, which gives that 
\begin{align*} 
  (\omega^{(c)}_J)''(r)< - \frac{(2J-2)!}{(J-2)!(J-2)!2^{2J-5}}\left( \frac{r}{c}\right)^{2J-2},  \ \text{ for } \ r>c.
\end{align*}
We conclude that $(\omega^{(c)}_J)''$ has a unique zero $r_2\in(c/2,c)$.



Similarly, one can easily check that 
\begin{align*}
  (\omega^{(c)}_J)'(0)=0 \text{ and }   (\omega^{(c)}_J)(c)>0 \text{ and } (\omega^{(c)}_J)'(2c)<0.
\end{align*}
Then, since $(\omega^{(c)}_J)''(0)$ equals $1$ at $0$, is decreasing, and has the unique zero at $r_2$,
$(\omega^{(c)}_J)'$ is increasing on $(0,r_2)$ and decreasing on $(r_2,\infty)$. Thus, we conclude that $(\omega^{(c)}_J)'$ has a unique zero $r_1\in(c,2c)$.
Next, we compute that 
  \begin{align*} 
  r^{-1}(\omega^{(c)}_J)'(r)
  &=
  \sum_{j=1}^{J_0}\frac{(4j-4)!r^{4j-2}}{(2j-2)!(2j-2)!2^{4j-4}c^{4j-2}}
  \left(  \frac{c^2}{r^2} -\frac{4j-3}{4j-2}  \right).
 \end{align*}
Since $\frac12\le\frac{4j-3}{4j-2} <1$ for all $j\in\N$, 
if $r>2c$, all terms in the summation at the last line are negative, so we have 
\begin{align*} 
  r^{-1}(\omega^{(c)}_J)'(r)
 < - \frac{(2J-4)!}{(J-2)!(J-2)!2^{2J-5}} \left(\frac{r}{c} \right)^{2J-2} .
\end{align*}
On the other hand, for $r<c$, all terms are positive, so
\begin{align*} 
  r^{-1}(\omega^{(c)}_J)'(r)>   \left(\frac{r}{c} \right)^2 \left(  \frac{c^2}{r^2} -\frac{1}{2} \right) >\frac12.
\end{align*}
\color{black}
\end{proof}


  
 
  \color{black}
  We recall from Lemma~\ref{Lem:Hessian} that 
  \begin{equation*}
    \textup{det}\He(\Omega^{(c)}_J)(\xi)
    =(\omega^{(c)}_J)''(|\xi|)\big\{(\omega^{(c)}_J)'(|\xi|)|\xi|^{-1}\big\}^{d-1}.
  \end{equation*}
Thus, we check from the above lemma that the determinant of Hessian of $\Omega^{(c)}_J$ is zero if and only if $|\xi|=r_1$ or $r_2$.
Since $c/2<r_1,r_2<2c$, heuristically speaking, the phase function $\Omega^{(c)}_J$ for an even number $J$ could be degenerate provided that the norm of momentum is close to the speed of light.
Next, we characterize the rank of the Hessian matrix on each degenerate sphere.
  
  \begin{lemma}\label{Lem: even order rank}
    Let $J\in2\N$.
\begin{align}
  \textup{rank}\He(\Omega^{(c)}_J)(\xi) = \begin{cases}
    d-1 &\text{ for }  |\xi|=r_2, \\ 
    1 &\text{ for }  |\xi|=r_1.
  \end{cases}
\end{align}
    \end{lemma}
  \begin{proof}
  We show that $\He(\Omega^{(c)}_J)(\xi)=\textup{span}\{\xi\}$ for $|\xi|=r_1$ and $\He(\Omega^{(c)}_J)(\xi)=\textup{span}\{\xi^{\perp}\}$ for $|\xi|=r_2$, where $\xi^\perp$ denotes a vector orthogonal to $\xi$.
Given any nonzero $\eta\in\mathbb{R}^d$, we decompose $\eta=k\xi+\xi^\perp$ for some vector $\xi^\perp$ orthogonal to $\xi$ and $k\in\R$. Then, using the notations in \eqref{A and B}, we calculate 
$$\begin{aligned}
  \Big(\left[\He(\Omega^{(c)}_J)(\xi)\right]\eta\Big)_n&=\sum_{\ell=1}^d \left(A\delta_{n\ell}+B\frac{\xi_n\xi_\ell}{|\xi|^2}\right)\eta_\ell= A\eta_n+B\frac{\xi\cdot\eta}{|\xi|^2}\xi_n\\
  &=k(A+B)\xi_n+A(\xi^\perp)_n=k(\omega^{(c)}_J)''(|\xi|)\xi_n+(\omega^{(c)}_J)'(|\xi|)|\xi|^{-1} (\xi^\perp)_n.
  \end{aligned}$$
          Thus, on the smaller sphere $|\xi|=r_2$, $\left[\He(\Omega^{(c)}_J)(\xi)\right]\eta=0$ if and only if $\xi^\perp=0$ (or $\eta=k\xi$), whereas on the lager sphere $|\xi|=r_1$, $\left[\He(\Omega^{(c)}_J)(\xi)\right]\eta=0$ if and only if $k\xi=0$ (or $\eta=\xi^\perp$).
    \end{proof} 

From the view point of Section~2.2, the above lemma says that we can expect a time decay of solutions to \eqref{hLS} to be at most $t^{-\frac12}$ on the larger sphere $\{|\xi|=r_1\}$ and $t^{-\frac{d-1}{2}}$ on the smaller sphere $\{|\xi|=r_2\}$.
Also, we infer from  the proof of lemma that 
the Hessian of the phase function $\Omega_J^{(c)}$ restricted to radial (resp. perpendicular) direction is nondegenerate on larger (resp. smaller) sphere. 
First, let us verify this at a point on intersection of the sphere and axis along $\mathbf{e}_j$, the simplest case, where the differentiation with radial direction or perpendicular direction is easy to find
\begin{align*}
  \textup{det} \textbf{H}_{\{\mathbf{e}_j\}}(\Omega_J^{(c)})(r_1\mathbf{e}_j) = (\omega_J^{(c)})''(r_1),  \quad  
  \textup{det} \textbf{H}_{\{\mathbf{e}_1,\cdots,\mathbf{e}_{j-1},\mathbf{e}_{j+1},\cdots,\mathbf{e}_d\}}(\Omega_J^{(c)})(r_2\mathbf{e}_j) = r_2^{-1}(\omega_J^{(c)})'(r_2).
\end{align*}  
Then, we check from Lemma~\ref{Lem:even order phase function} that 
\begin{align*}
  \left |\textup{det} \textbf{H}_{\{\mathbf{e}_j\}}(\Omega_J^{(c)})(r_1\mathbf{e}_j)\right| \ge \frac12,  \quad  
\left|\textup{det} \textbf{H}_{\{\mathbf{e}_1,\cdots,\mathbf{e}_{j-1},\mathbf{e}_{j+1},\cdots,\mathbf{e}_d\}}(\Omega_J^{(c)})(r_2\mathbf{e}_j)\right| \ge \frac12,
\end{align*} 
where both lower bounds are independent of $c$.
In the next lemma, we prove that the determinants of the above submatrices of Hessian at points near the each degenerate sphere and around $\mathbf{e}_j$ direction also have the uniform-in-$c$ lower bound.
In addition, we prove that outside the degenerate spheres the determinant of Hessian matrix has the uniform-in-$c$ lower bound.
\begin{lemma}\label{Lem:Hessian lower bound even}
  Let $J\in 2\N$ and $\xi\in\R^d$.
  There exist $0<\delta\ll1$ independent of $c$ such that 
\begin{align}\label{delta}
  \min(c-r_2, r_1-c) \ge \delta.
\end{align}
Then, we have 
\begin{align}   \label{Hessian lower bound even}  
   \left| \textup{det}\He(\Omega_J^{(c)}) (\xi) \right| 
   \gs \left( 1+\tfrac{|\xi|}{c}\right)^{(2J-2)d}
   \text{ if } ||\xi|-r_1|> c\delta \text{ or } ||\xi|-r_2|> c\delta,
\end{align}
where the implicit constant is independent of $c$.
Furthermore, if $\xi\in\Theta_j$ for some $1\le j \le d$, one has 
\begin{align}
  \left|  \textup{det} \He_{\{\mathbf{e}_j\}}(\Omega_J^{(c)})(\xi)\right| &\ge \frac{1}{4d}|(\omega_J^{(c)})''(\xi)| \ge \frac{1}{8d},   &\text{ if }||\xi|-r_1|\le 2c\delta, \label{hessian r1}\\ 
  \left|\textup{det} \He_{\{\mathbf{e}_1,\cdots,\mathbf{e}_{j-1},\mathbf{e}_{j+1},\cdots,\mathbf{e}_d\}}(\Omega_J^{(c)})(\xi)\right| &\ge \frac{1}{4d} \Big( \frac{|(\omega_J^{(c)})'(\xi)|}{|\xi|}\Big)^{d-1} \ge \frac{1}{8d}, &\text{ if }||\xi|-r_2|\le 2c\delta.\label{hessian r2} 
\end{align}
\end{lemma}

\begin{proof}
  \textit{(1). Proof of \eqref{delta}.}
  By applying Mean value theorem to $(\omega_J^{(c)})''$, we have 
  \begin{align*} 
    (\omega_J^{(c)})''(c)= (\omega_J^{(c)})''(c)- (\omega_J^{(c)})''(r_2)= (\omega_J^{(c)})'''(r_*)(c-r_2) 
    \text{ for some } r_*\in (r_2,c).
  \end{align*}
Then, from \eqref{Taylor omega3 less than c},  
  \begin{align*} 
  \frac12 \le | (\omega_J^{(c)})''(c)| =| (\omega_J^{(c)})'''(r_*)||c-r_2| 
  \le 3\cdot 2^{-\frac52}c^{-1}|c-r_2|,
  \end{align*}
  which gives that 
  \begin{align*}
   c-r_2 \ge  3^{-1}2^{\frac32}c.
  \end{align*}
  Similarly, 
  we have 
  \begin{align*} 
    \frac{c}{2}\le | (\omega_J^{(c)})'(c)|= | (\omega_J^{(c)})''(r_\star)||r_1-c| 
    \le | (\omega_J^{(c)})''(2c)||r_1-c|,
  \end{align*}
  for some $r_\star\in (c,r_1)$, which gives that 
\begin{align*} 
  r_1-c \ge \tfrac{1}{2}| (\omega_J^{(c)})''(2c)|^{-1}c.
 \end{align*}
Here, we observe that $| (\omega_J^{(c)})''(2c)|$ is independent of $c$.

\textit{(2). Proof of \eqref{Hessian lower bound even}}. 
Recall from \eqref{radial Hessian} that 
\begin{align*}
  \textup{det}\He (\Omega_J^{(c)})(\xi)
  = (\omega_J^{(c)})''(|\xi|)\big\{ (\omega_J^{(c)})'(|\xi|)|\xi|^{-1}\big\}^{d-1},
\end{align*}
so, it can be easily shown from Lemma~\ref{Lem:even order phase function} that 
\begin{align}\label{det lowbound}
  |\textup{det}\He (\Omega_J^{(c)})(\xi)| \gs\left( 1+\tfrac{|\xi|}{c}\right)^{(2J-2)d} \text{ for } |\xi|\le \tfrac c2 \text{ or } |\xi|\ge 2c.
\end{align}
Next, consider $\frac{c}{2}\le |\xi| \le 2c$.
If $r\in(\frac c2,2c)$ and $|r-r_2|\ge c\delta$, by Mean value theorem, we have
\begin{align*} 
 |(\omega_J^{(c)})''(r)|&=|(\omega_J^{(c)})''(r_2)+(\omega_J^{(c)})'''(r_*) (r-r_2)| \text{ for some } r_* \text{ between } r \text{ and }r_2 \\
 &=|(\omega_J^{(c)})'''(r_*)| |r-r_2| \ge 3\cdot2^{-\frac72}\delta,
\end{align*}
where we used (1) in Lemma~\ref{Lem:even order phase function}.
Also, if $r\in(\frac c2,2c)$ and $|r-r_1|\ge c\delta$,
\begin{align*} 
 |(\omega_J^{(c)})'(r)| &=|(\omega_J^{(c)})'(r_1) + (\omega_J^{(c)})''(r_\star )(r-r_1)| \text{ for some } r_\star  \text{ between } r \text{ and }r_1 \\ 
 &=|(\omega_J^{(c)})''(r_{\star} )||r-r_1| 
 \ge \frac12 c\delta,
\end{align*}
which gives that 
\begin{align*} 
 |r^{-1}(\omega_J^{(c)})'(r)| \ge \tfrac{1}{4}\delta.
\end{align*}
By applying two inequalities to the above  formula for determinant with $|\xi|=r$, we obtain 
for $\frac{c}{2}\le |\xi| \le 2c$ with $||\xi|-r_1|\ge c\delta$ or $||\xi|-r_2|\ge c\delta$ that 
\begin{align*} 
  |\textup{det}\He (\Omega_J^{(c)})(\xi)|\ge 3\cdot2^{-\frac72-2(d-1)}\delta^d. 
\end{align*}



  \textit{(3). Proof of \eqref{hessian r1}}. 
A direct computation gives that  
\begin{align*} 
 \partial_{\mathbf{e}_j}^2 \Omega_J^{(c)}(\xi) 
  = (\omega_J^{(c)})''(|\xi|)\frac{\xi_j^2}{|\xi|^2} 
  + (\omega_J^{(c)})'(|\xi|)|\xi|^{-1} \frac{|\check{\xi}_j|^2}{|\xi|^2}.
\end{align*}  
We claim that the latter term on the right-hand side is dominated by the former if $\delta$ is sufficiently small. Indeed, by Mean value theorem, we have 
\begin{align*} 
  (\omega_J^{(c)})'(r)=(\omega_J^{(c)})'(r)-(\omega_J^{(c)})'(r_1)=(\omega_J^{(c)})''(r_*)(r-r_1), \text{ for some } r_1< r_*<r.
 \end{align*}
 Since $r_1\in(c,2c)$ and $(\omega_J^{(c)})''$ is decreasing and negative on $(c,\infty)$ we have 
 $|(\omega_J^{(c)})''(r_*)|< |(\omega_J^{(c)})''(r)|$, which gives 
 $|(\omega_J^{(c)})'(r)| \le |r-r_1||(\omega_J^{(c)})''(r)|$.
 Then, for $\xi\in\Theta_j$ and $||\xi|-r_1|\le c\delta$, we have 
 \begin{align*}
  \left| (\omega_J^{(c)})'(|\xi|)|\xi|^{-1} \frac{|\check{\xi}_j|^2}{|\xi|^2}\right|
  \le \frac{2d}{2d-1}|(\omega_J^{(c)})''(|\xi|)|\frac{||\xi|-r_1|}{|\xi|} 
  \le \frac{4d}{2d-1}\delta |(\omega_J^{(c)})''(|\xi|)|.
 \end{align*}
Since $\left|(\omega_J^{(c)})''(|\xi|)\frac{\xi_j^2}{|\xi|^2}\right| \ge \frac{1}{2d} |(\omega_J^{(c)})''(|\xi|)| $, by taking $\delta$ sufficiently small, we can show \eqref{hessian r1}.

    \textit{(4). Proof of \eqref{hessian r2}}. 
Using the notations in \eqref{A and B}, we compute  
   \begin{align}\begin{aligned}\label{GE}
    \He_{\{\mathbf{e}_1,\cdots,\mathbf{e}_{j-1},\mathbf{e}_{j+1},\cdots,\mathbf{e}_d\}}(\Omega_J^{(c)})(\xi) 
     &=\begin{bmatrix}
    A\mathbf{e}_1+(B\frac{\xi_1}{|\xi|^2})\check\xi_j\\
    \vdots\\
    A\mathbf{e}_{i-1}+(B\frac{\xi_{j-1}}{|\xi|^2})\check\xi_j\\
    A\mathbf{e}_{i+1}+(B\frac{\xi_{j+1}}{|\xi|^2})\check\xi_j\\
    \vdots\\
    A\mathbf{e}_{d}+(B\frac{\xi_d}{|\xi|^2})\check\xi_j
  \end{bmatrix}.
     \end{aligned}\end{align} 
By performing Gaussian elimination as in Lemma~\ref{Lem:Hessian}, we obtain    
\begin{align*}
  \textup{det}\He_{\{\mathbf{e}_1,\cdots,\mathbf{e}_{j-1},\mathbf{e}_{j+1},\cdots,\mathbf{e}_d\}}(\Omega_J^{(c)})(\xi) =
   \left( \frac{(\omega_J^{(c)})'(|\xi|)}{|\xi|} \right)^{d-2}
\left\{ 
  \left(\frac{(\omega_J^{(c)})'(|\xi|)}{|\xi|}\right)\frac{\xi_j^2}{|\xi|^2}+(\omega_J^{(c)})''(|\xi|)\frac{|\check\xi_j|^2}{|\xi|^2}\right\}.
\end{align*}
We show that the first term of the right-hand side is dominant. By Mean value theorem and \eqref{bound for 3 omega} we have for $r\in(c/2,c)$, 
  \begin{align*} 
   |(\omega_J^{(c)})''(r)|&=|(\omega_J^{(c)})'''(r_*)||r-r_2| \text{ for some } r_* \text { between } r \text{ and } r_2 \\ 
   &\le 3\cdot 2^{-\frac52}c^{-2} r_*|r-r_2| 
   \le 3\cdot 2^{-\frac32}\delta.
  \end{align*}
from which we estimate the second term of the right-hand side that for $\xi\in\Theta_i$ 
\begin{align*}
  \left| (\omega_J^{(c)})''(|\xi|)\frac{|\check\xi_j|^2}{|\xi|^2}\right|
  \le \frac{2d}{2d-1} |(\omega_J^{(c)})''(|\xi|)| \le \frac{2d}{2d-1} 3\cdot 2^{-\frac32}\delta.
\end{align*} 
On the other hand,   for $\xi\in \Theta_j$ and  $|\xi|\le c$
\begin{align*}
  \left|   \left( \frac{(\omega_J^{(c)})'(|\xi|)}{|\xi|} \right) \frac{\xi_j^2}{|\xi|^2} \right|
  \ge  \frac{1}{2d} \left|   \frac{(\omega_J^{(c)})'(|\xi|)}{|\xi|}  \right| 
  \ge  \frac{1}{4d} \left|   \frac{(\omega_J^{(c)})'(|\xi|)}{|\xi|}  \right| + \frac{1}{8d}.
\end{align*}
Thus, by making $\delta$ sufficiently small if necessary, we can obtain \eqref{hessian r2}.
\end{proof}

\subsection{Dispersive estimates}

Using the projection operator \eqref{projection operator}, we write the linear solution to \eqref{hLS} as
\begin{align*}
  e^{-it\mathcal{H}_J^{(c)}}\psi_0(x) 
  &=\sum_{N\in \Z} 
  e^{-it\mathcal{H}_J^{(c)}}P_N\psi_0(x) 
  =\sum_{N\in\Z}
  \mathcal{I}_{J,N}^{(c)}\left(t,\tfrac{\cdot}{t}\right)\ast \psi_0(x),
\end{align*}
where the kernel is given by 
\begin{align}\label{frequency localized kernel}
  \mathcal{I}_{J,N}^{(c)}(t,v)  =\frac{1}{(2\pi)^d}\int_{\R^d} 
  e^{it(v\cdot\xi-\Omega_J^{(c)}(\xi))} \chi_N(\xi)d\xi.
\end{align}
We investigate the time decay of $\mathcal{I}_{J,N}^{(c)}(t,v)$ by using the results in the previous subsection.

\subsubsection{Multi dimensional case $d\ge2$}\label{subsection:Dispersive estimates}
\begin{proposition}\label{Prop:even dispersive}
  Let $J\in 2\N$ and $N\in \Z$.
For $d\ge2$, we have 
\begin{align} \label{even dispersive estimates multi d}
    \sup_{v\in\R^d}\left| \mathcal{I}_{J,N}^{(c)}(t,v)\right|
    \ls   
    \begin{cases}
      t^{-\frac d2} &\text{ for } \ 2^N < c/4 \text{ or } 2^N>4c, \\ 
    c^{d-1}t^{-\frac12}  &\text{ for } \ c/4 \le 2^N \le 4c,    
    \end{cases}
\end{align}
where the implicit constants are independent of $c\ge1$.
\end{proposition}

\begin{remark}
We do not claim the sharpness of time decay when $c/4 \le 2^N \le 4c$.
\end{remark}

\begin{proof}
For $2^N < c/4 \text{ or } 2^N>4c$ when the degenerate points are excluded, $\textup{det}\He(\Omega_J^{(c)})$ has the same lower bound \eqref{Hessian lower bound even} as in the odd case (see Proposition~\ref{odd hessian}), so by applying Lemma~\ref{Lem:Stationary phase method} we can show that 
\begin{align*}
  | \mathcal{I}_{J,N}^{(c)}(t,v)|\ls t^{-\frac d2}.
\end{align*}

Next, we consider the case when $c/4\le 2^N\le 4c$.
Let us decompose the support of integral $\mathcal{I}_{J,N}^{(c)}$ into 
\begin{align*} 
  \mathcal{I}_{J,N}^{(c)} &= \mathcal{K}_1+\mathcal{K}_2+\mathcal{K}_3, \\
  \mathcal{K}_1&=  \frac{1}{(2\pi)^d}
\int_{\R^d}  e^{it\big(v\cdot\xi-\Omega_J^{(c)}(\xi)\big)}\chi\left(\tfrac{|\xi|-r_1}{c\delta}\right)\chi_N(\xi) d\xi, \\
\mathcal{K}_2&=\frac{1}{(2\pi)^d}
\int_{\R^d}  e^{it\big(v\cdot\xi-\Omega_J^{(c)}(\xi)\big)}
\left( 1- \chi\left(\tfrac{|\xi|-r_1}{c\delta}\right) \right)
\chi\left(\tfrac{|\xi|-r_2}{c\delta}\right)\chi_N(\xi) d\xi, \\
\mathcal{K}_3&=\frac{1}{(2\pi)^d}
\int_{\R^d}  e^{it\big(v\cdot\xi-\Omega_J^{(c)}(\xi)\big)}
\left(1-\chi\left(\tfrac{|\xi|-r_1}{c\delta}\right) \right) 
\left(1-\chi\left(\tfrac{|\xi|-r_2}{c\delta}\right) \right) 
\chi_N(\xi) d\xi,
\end{align*}
for $\delta\ll1$ chosen in the Lemma~\ref{Lem:Hessian lower bound even}. 
We observe that  $\mathcal{K}_1$ is supported near the larger sphere $\{|\xi|=r_1\}$, $\mathcal{K}_2$ is near the smaller sphere $\{|\xi|=r_2\}$ and $\mathcal{K}_3$ is outside two spheres.
Since $\delta$ in Lemma~\ref{Lem:Hessian lower bound even} was chosen independent of $c$, even though all constants in the estimates below might be dependent on $\delta$, we do not write them explicitly and focus on $c$.

$\mathcal{K}_3$ can be estimated similarly as the above case when $2^N<\frac{c}{4}$ or $2^N>4c$ thanks to \eqref{Hessian lower bound even} 
\begin{align*} 
  |\mathcal{K}_3| \ls t^{-\frac d2}.
\end{align*}
Interpolating this with the trivial bound $|\mathcal{K}_3| \ls c^d$, we can obtain the desired result.

Now, we consider $\mathcal{K}_1$ and $\mathcal{K}_2$ the supports of which contain the degenerate points.
Applying the partition of unity \eqref{partition of unity}, we write 
$ \mathcal{K}_j=\sum_{\ell=1}^d  \mathcal{K}_j^\ell$ for $j=1,2$ where 
\begin{align*} 
  \mathcal{K}_1^\ell &=  \frac{1}{(2\pi)^d}
 \int_{\R^d}  e^{it\big(v\cdot\xi-\Omega_J^{(c)}(\xi)\big)}\chi\left(\tfrac{|\xi|-r_1}{c\delta}\right)\chi_N(\xi)\theta^\ell(\xi) d\xi, \\ 
 \mathcal{K}_2^\ell &=  \frac{1}{(2\pi)^d}
 \int_{\R^d}  e^{it\big(v\cdot\xi-\Omega_J^{(c)}(\xi)\big)}\left( 1- \chi\left(\tfrac{|\xi|-r_1}{c\delta}\right) \right)
 \chi\left(\tfrac{|\xi|-r_2}{c\delta}\right)\chi_N(\xi) \theta^\ell(\xi) d\xi.
\end{align*}
For $ \mathcal{K}_2^\ell$, we change variables as
\begin{align*}
  \mathcal{K}_2^\ell &=  \frac{1}{(2\pi)^d}
  \int_{\R^d}  e^{it\big(v\cdot\xi-\Omega_J^{(c)}(\xi)\big)}\left( 1- \chi\left(\tfrac{|\xi|-r_1}{c\delta}\right) \right)
  \chi\left(\tfrac{|\xi|-r_2}{c\delta}\right)\chi_N(\xi) \theta^\ell(\xi) d\xi    \\ 
  &= \frac{2^{Nd}}{(2\pi)^d}
  \int_{\R^d}  e^{i2^{2N}t\big(2^{-N}v\cdot\xi-2^{-2N}\Omega_J^{(c)}(2^N\xi)\big)}\left( 1- \chi\left(\tfrac{|2^N\xi|-r_1}{c\delta}\right) \right)
  \chi\left(\tfrac{|2^N\xi|-r_2}{c\delta}\right)\chi_0(\xi) \theta^\ell(\xi) d\xi. 
\end{align*}
We claim that the integral with $\check{\xi_\ell}$ variables is bounded as follows:
\begin{align}
  \begin{aligned}\label{int xi check}
  & \left| \int_{\R^{d-1}} e^{i2^{2N}t\big(2^{-N}v\cdot\xi-2^{-2N}\Omega_J^{(c)}(2^N\xi)\big)}\left( 1- \chi\left(\tfrac{|2^N\xi|-r_1}{c\delta}\right) \right)
  \chi\left(\tfrac{|2^N\xi|-r_2}{c\delta}\right)\chi_0(\xi) \theta^\ell(\xi) d\check{\xi_\ell} \right| \\ 
  &\hspace{28.5em}\ls (2^{2N}t)^{-\frac{d-1}{2}},
  \end{aligned}\end{align}
where the implicit constant is independent of $c$.  
For $c/4\le 2^N\le 4c$, we have from \eqref{hessian r2}
\begin{align*}
  \left|  \textup{det} \He_{\check{\xi}_\ell}\left[ (2^{-2N}\Omega_J^{(c)})(2^N\xi) \right] \right| \gs 1, 
\end{align*}
and we can show by direct computations that 
\begin{align*}
& \sup_{2\le|\alpha|\le d+1}  \Big\{
    C_\alpha : \sup_{\xi \in \textup{supp} \chi}\big|{\partial_{\check{\xi}_\ell}}^\alpha
  \big\{ (2^{-N}v\cdot\xi-2^{-2N}\Omega_J^{(c)}(2^N\xi)) \big\} \big|\le C_\alpha, \Big\}
  \ls 1, \\ 
 & \sup_{|\alpha|\le d}  \Big\{
    C_\alpha : \sup_{\xi }\left|
    {\partial_{\check{\xi}_\ell}}^\alpha   \left\{ 
    \left( 1- \chi\left(\tfrac{|2^N\xi|-r_1}{c\delta}\right) \right)
    \chi\left(\tfrac{|2^N\xi|-r_2}{c\delta}\right)\chi_0(\xi) \theta^\ell(\xi) 
    \right\} \right|
    \le C_\alpha, \Big\}
  \ls 1,
\end{align*}
where  the implicit constants are independent of $v\in\R^d$, $N$, and $c$. Then, \eqref{int xi check} can be derived by applying Lemma~\ref{Lem:Stationary phase method} with the help of above estimates.Since $\xi_\ell$ belongs to the interval of length at most $\frac{c}{2^N}$, we conclude that
\begin{align*}
  \mathcal{K}_2^\ell &\ls  2^{Nd} \int_{\R} \left| \int_{\R^{d-1}} e^{i2^{2N}t\big(2^{-N}v\cdot\xi-2^{-2N}\Omega_J^{(c)}(2^N\xi)\big)}\left( 1- \chi\left(\tfrac{|2^N\xi|-r_1}{c\delta}\right) \right)
  \chi\left(\tfrac{|2^N\xi|-r_2}{c\delta}\right)\chi_1(\xi) \theta^\ell(\xi) d\check{\xi_\ell} \right|
  d\xi_\ell  \\ 
  &\ls 2^{Nd} \cdot (2^{2N}t)^{-\frac{d-1}{2}} \cdot \frac{c}{2^N}
  \ls t^{-\frac {d-1}{2}}c.
\end{align*}
Finally, by interpolating this with the trivial bound $  |\mathcal{K}_2^\ell| \ls c^d$, we obtain the desired result.

For $\mathcal{K}_1^\ell$, the strategy is similar. We first use Lemma \ref{Lem:Stationary phase method} in terms of $\xi_{\ell}$, and measure the set of $\check{\xi_\ell}$ variables.
To be specific, since 
$$\xi\in\Theta_\ell \text{ and } ||\xi|-r_1|\le c\delta \text{ on the support of integral } \mathcal{K}_1^\ell,$$ 
we have the uniform-in-$c$ lower bound of the Hessian in terms of $\xi_\ell$ variable, i.e., \eqref{hessian r1}, which implies by Lemma~\ref{Lem:Van} that  

\begin{align*}
  \left| \int_{\R}  e^{it\big(v\cdot\xi-\Omega_J^{(c)}(\xi)\big)}\chi\left(\tfrac{|\xi|-r_1}{c\delta}\right)\chi_N(\xi)\theta^\ell (\xi) d\xi_\ell \right|    
  \ls t^{-\frac12},
\end{align*}
where the implicit constant is independent of $c$. Since $\check{\xi}_\ell$ belongs to the $(d-1)$ dimensional sphere of volume at most $c^{d-1}$,we conclude that
\begin{align*}
  |\mathcal{K}_1^\ell| &\ls 
  \int_{\R^{d-1}} \left| \int_{\R}  e^{it\big(v\cdot\xi-\Omega_J^{(c)}(\xi)\big)}\chi\left(\tfrac{|\xi|-r_1}{c\delta}\right)\chi_N(\xi)\theta^\ell(\xi) d\xi_\ell \right| d\check{\xi_\ell}  \\ 
  &\ls t^{-\frac12}c^{d-1}.
\end{align*} 

\color{black}
\end{proof}

\subsubsection{One dimensional case $d=1$}
\begin{proposition}\label{Prop:dispersive 1d}
  Let $J\in 2\N$ and $N\in \Z$. For $d=1$, we have 
  \begin{align} \label{odd dispersive estimates 1d}
    \sup_{v\in\R}\left| \mathcal{I}_{J,N}^{(c)}(t,v)\right| \ls   
    \begin{cases} 
      \  t^{-\frac 12} &\text{ for }  2^N<c/4 \text{ or } 2^N>4c,\\
      \  c^{\frac13}t^{-\frac13}   &\text{ for }   c/4\le 2^N \le 2c,
      \end{cases}
    \end{align}
  where the implicit constants are independent of $c>1$.
\end{proposition}
\begin{proof}
 We observe that one of the second or third derivative 
 of the phase function 
 has a lower bound over each dyadic pieces. 
  More precisely, 
  we have from Lemma~\ref{Lem:even order phase function} that 
  \begin{align*}
  |(\omega^{(c)}_J)''(r)|\ge b \text{ for } r\in \textup{supp}\chi_N,  \ 2^N<c/4 \text{ or } 2^N>4c.
  \end{align*}
which implies by Lemma~\ref{Lem:Van} that 
  \begin{align*}
  | \mathcal{I}_{J,N}^{(c)}(t,v)| \
   \ls  t^{-\frac 12} \text{ for } \
   2^N<c/4 \text{ or } 2^N>4c,
  \end{align*}
  where the implicit constant is independent of $c$.
  Next, we have from \eqref{bound for 3 omega} in Lemma~\ref{Lem:even order phase function} that
  \begin{align*}
  |(\omega^{(c)}_J)'''(r)|\ge 3c^{-1}\cdot2^{-\frac92}, \text{ for } r\in \textup{supp}\chi_{N}, \ c/4\le 2^N \le 2c
  \end{align*}
  which implies again by Lemma~\ref{Lem:Van} that
  \begin{align*}
    | \mathcal{I}_{J,N}^{(c)}(t,v)| \ls  c^\frac13t^{-\frac 13} \text{ for } \ c/4\le 2^N \le 2c .
  \end{align*}
\end{proof}

\subsection{Proof of Theorem~\ref{Thm:Strichartz even hLS}}
We only consider multi dimensional case $d\ge2$. A slight modification would give the proof for one dimensional case. From \eqref{even dispersive estimates multi d},
we can obtain, by applying the $TT^*$ argument in Keel-Tao \cite{KT1998}, the frequency localized Strichartz estimates are given by
  \begin{align*}
  \Big\| e^{-it\mathcal{H}_J^{(c)}}P_N\psi_0\Big\|_{L_t^qL_x^r} 
  &\le A c^{\frac{2(d-1)}{q}} \|\psi_0\|_{L^2(\R^d)} &\text{ for } c/4\le 2^N \le 4c,\\
  \Big\| e^{-it\mathcal{H}_J^{(c)}}P_N\psi_0\Big\|_{L_t^{\tilde q}L_x^{\tilde r}} 
  &\le  A \|\psi_0\|_{L^2(\R^d)} &\text{ for } 2^N<c/4 \text{ or } 2^N>4c,
 \end{align*}
  which holds for all pair $(q,r),(\tilde q,\tilde r)$ satisfying $2\le q,r,\tilde q,\tilde r\le\infty$ and $\frac{2}{q}+\frac{1}{r}=\frac{1}{2}$ 
  and $\frac{2}{\tilde q}+\frac{d}{\tilde r}=\frac{d}{2}$ with $r\neq\infty$ and $(\tilde q,\tilde r,d)\neq (2,\infty,2)$.
Then, for $c/4\le 2^N\le 4c$ we have 
\begin{align*} 
  \Big\| e^{-it\mathcal{H}_J^{(c)}}P_N\psi_0\Big\|_{L_t^qL_x^r} 
  \le A c^{\frac{2(d-1)}{q}} \|P_N\psi_0\|_{L^2(\R^d)} 
  \le A \||\nabla|^\frac{2(d-1)}{q} P_N\psi_0\|_{L^2(\R^d)}.
\end{align*}
If $2^N<c/4$ or $2^N>4c$, we apply the Sobolev embedding to obtain
\begin{align*} 
  \Big\| e^{-it\mathcal{H}_J^{(c)}}P_N\psi_0\Big\|_{L_t^qL_x^r} 
  \le 
  \Big\||\nabla|^{\left(\frac{d}{\tilde r}-\frac dr\right)}e^{-it H_c^J(-\Delta)} P_N\psi_0\Big\|_{L_t^q L_x^{\tilde r}} 
  &\le A \| |\nabla|^\frac{2(d-1)}{q} P_N\psi_0 \|_{L^2(\R^d)},
\end{align*}
where $\frac{2}{q}+\frac{d}{\tilde r}=\frac d2$.
Then, by the Littlewood-Paley inequality and Minkowski inequality with $q,r\ge2$, we obtain that
\begin{align*}
&\| e^{-it\mathcal{H}_J^{(c)}}P_N\psi_0\|_{L_t^qL_x^r}
\approx \| \{ \sum_{N\in\Z} |e^{-it H_{c}^J(-\Delta)} P_N\psi_0|^2\}^\frac12 \|_{L_t^qL_x^r} \\
&\ls  \left\{ \sum_{N\in\Z} \| e^{-itH_{c}^J(-\Delta)} P_N\psi_0\|_{L_t^qL_x^r}^2  \right\}^\frac12 
\ls \left\{ \sum_{N\in\Z} \| |\nabla|^\frac{2(d-1)}{q}  P_N\psi_0\|_{L^2}^2  \right\}^\frac12 
\approx \| \psi_0\|_{ \dot H^\frac{2(d-1)}{q}}^2.
\end{align*}  

\color{black}
   

\section{Application 1: Non-relativistic approximation of Hartree-Fock or Hartree equations}
As applications of Strichartz estimates for higher-order linear Schr\"odinger equation, we study the non-relativistic limit problems for nonlinear equations with Hartree-Fock or Hartree nonlinearities when $J$ is odd. Here, we only deal with Hartree-Fock nonlinearity, because analogous argument can be directly applied to Hartree nonlinearity. Recall the pseudo-relativistic Hartree-Fock equation (with $\hbar, m = 1$)
\begin{equation}\label{pNLHF-1}
    i\partial_t \psi_k^{(c)} =\mathcal{H}^{(c)}\psi_k^{(c)}+H(\psi_k^{(c)})-F_k(\psi_k^{(c)}),\quad k=1,2,..., N,
\end{equation}
where 
$$H(\psi_k^{(c)})=\sum_{\ell=1}^N\left(\frac{\kappa}{|x|}*|\psi_\ell^{(c)}|^2\right)\psi_k^{(c)} \quad \mbox{and} \quad F_k(\psi_k^{(c)})=\sum_{\ell=1,\ell\neq k}^N\left(\frac{\kappa}{|x|}*(\overline{\psi_\ell^{(c)}}\psi_k^{(c)})\right)\psi_\ell^{(c)}.$$

\subsection{ Uniform bound for nonlinear solutions}
\subsubsection{Results on psuedo-relativistic equations}\label{sub:Results on pNLHF}
The well-posedness results for \eqref{pNLHF-1} on the energy space $H^\frac12(\R^3)$ have been established in \cite{Cho2006, FL-2007, Lenzmann2007}. The proof of the local well-posedness follows from the standard contraction mapping principle with Sobolev embedding, and this local solution can be extended to the global one via conservation laws of the energy and mass. We emphasize that global bounds of the solutions are independent of $c$. The statement of the result (without proof) is given below:
\begin{proposition}\label{Pro:HRNLS}
 Let $c\ge1$ and $\Psi_0 \in  H^\frac12(\R^3;\C^N)$ be given initial data. Suppose that $\kappa > 0$, or $\kappa <0$ and $\|\Psi_0\|_{ L^2(\R^3;\C^N)}$ is sufficiently small. Then, 
there exists a global solution $\Psi^{(c)} \in C(\R: H^\frac12(\R^3;\C^N))$ to \eqref{pNLHF-1} such that 
\begin{align}\label{Uniform bound} 
 \sup_{t\in \R} \| \Psi^{(c)}(t) \|_{H^\frac12(\R^3;\C^N) }  \le C(M,E),
\end{align}
where the bound $C(M,E)$ depends on mass and energy given by
\[\begin{aligned}
&M = M(\Psi_0) = \sum_{k=1}^N \|\psi_{k,0}\|_{L^2(\R^3)}^2 \quad \mbox{and} \\
&E = E(\Psi_0) =  \sum_{k=1}^N  \Big\{\frac12 \Big\la \psi_{k,0}, \Big\{\sqrt{c^4-c^2\Delta} - c^2m\Big\}\psi_{k,0}   \Big\ra  \\
&~{} \hspace{6em} + \frac{1}{4}\Big\la \sum_{\ell=1}^N\Big(\frac{\kappa}{|x|}\ast |\psi_{\ell,0}|^{2}\Big)\psi_{k,0} - \sum_{\ell=1}^N\Big(\frac{\kappa}{|x|}\ast \psi_{k,0}\overline{\psi_{\ell,0}})\Big)\psi_{\ell,0}, \psi_{k,0} \Big\ra \Big\}. \end{aligned}\]
\end{proposition}

\begin{remark}
A key estimate to handle Hartree-type nonlinearities is as follows (see  \cite[Lemma 3.2]{Choi2016} for the details of the proof): for $f_j \in H^s(\R^3)$, $j=1,2,3$,
\begin{align}\label{Estimates for nonlinearity} 
  \left\| \left(\frac{1}{|x|}\ast (f_1 f_2)\right)f_3\right\|_{H^s(\R^3)}   \ls \prod_{j=1}^3 \| f_j \|_{ H^s(\R^3) }, \quad \text{ for } s\ge\frac12,
 \end{align}
which follows from endpoint Sobolev inequality 
\begin{align}\label{Sobolev inequality} 
  \left\|\frac{1}{|x|}\ast (f_1f_2)\right\|_{L^\infty(\R^3)} 
  \ls \|f_1\|_{H^s(\R^3)}\|f_2\|_{H^s(\R^3)} \quad 
\text{ for } s\ge\frac12
\end{align}
and the fractional Leibniz rule \cite{Kato1995}. Moreover, the particular case when $s=0$ can be obtained by using the Sobolev inequality 
\begin{align}\label{nonlinear term in L2} 
  \left\| \left(\frac{1}{|x|}\ast (f_1f_2)\right)f_3\right\|_{L^2(\R^3)} 
  \ls  \| f_1 \|_{L^{3}(\R^3)}   
  \|f_2\|_{L^{2}(\R^3)}\| f_3\|_{L^{3}(\R^3)}.
  \end{align} 
\end{remark}
We are now focus on the higher-order equation with the Hartree-Fock nonlinearity:
\begin{equation}\label{hNLS-1}
    i\partial_t \phi_k^{(c)} =\mathcal{H}_J^{(c)}\phi_k^{(c)}+H(\phi_k^{(c)})-F_k(\phi_k^{(c)}),\quad k=1,2,..., N.
\end{equation}
%
The local well-posedenss of \eqref{hNLS-1} in $H^s(\R^3)$, $s\ge\frac12$ follows from the same argument in the proof of Proposition~\ref{Pro:HRNLS}. When $J$ is odd, this local solution can be extended to the global one in $H^1(\R^3)$ thanks to the conservation of the mass and energy: 

\[\begin{aligned}
&\mathcal{M}(\Phi_0)= \|\Phi_0\|_{L^2(\R^3;\C^N)}^2\\
&\mathcal{E}(\Phi_0)= \sum_{k=1}^N\Big\{\Big\la \phi_{k,0}, \sum_{j=0}^J\frac{(-1)^j(2j)!}{(j+1)!j!2^{2j+1}c^{2j}}(-\Delta)^{j+1}\phi_{k,0}\Big\ra\\
&~{} \hspace{6em} + \frac{1}{4}\Big\la \sum_{\ell=1}^N\Big(\frac{\kappa}{|x|}\ast |\phi_{\ell,0}|^{2}\Big)\phi_{k,0} - \sum_{\ell=1}^N\Big(\frac{\kappa}{|x|}\ast \phi_{k,0}\overline{\phi_{\ell,0}})\Big)\phi_{\ell,0}, \phi_{k,0} \Big\ra \Big\}.
\end{aligned}\]
%
Indeed, the kinetic part in the energy controls $\dot{H}^1(\R^3)$ norm of solutions 
\begin{align*} 
  \left\la \phi_{k,0}, \sum_{j=0}^J\frac{(-1)^j(2j)!}{(j+1)!j!2^{2j+1}c^{2j}}(-\Delta)^{j+1}\phi_{k,0}\right\ra
  &=\int \omega_c^J(|\xi|)|\widehat{\phi}_{k,0}(\xi)|^2 d\xi  \\
  &\ge \frac14 \int |\xi|^2|\widehat{\phi}_{k,0}(\xi)|^2 d\xi
  \gs  \|\phi_{k,0}\|_{\dot H^1}^2. 
\end{align*}
Thus, the solution $\Phi^{(c)}$ exists for all time with the following bound
\begin{align*} 
 \sup_{t\in \R}\| \Phi^{(c)}(t) \|_{H^1(\R^3;\C^N)} \le C(\mathcal{M},\mathcal{E}),
\end{align*}
independent of $c$.

\subsubsection{Uniform bounds of solutions to higher order equations}
In order to prove the convergence of \eqref{hNLS-1} globally in time for initial data given in $H^\frac12(\R^3;\C^N)$, the existence of global solutions to \eqref{hNLS-1}in $H^{\frac12}(\R^3;\C^N)$ is required. The following proposition guarantees that the global solutions to \eqref{hNLS-1} exist in $L^2(\R^3;\C^N)$, and it satisfies uniform-in-$c$ boundedness.
\begin{proposition}[Global $L^2$ solutions with uniform Strichartz bounds]\label{Pro:GlobalL2}
Let $c\ge1$ and $J\in 2\N-1$. For given $\Phi_0 \in L^2(\R^3;\C^N)$, there exists a unique global solution $\Phi^{(c)}$ to \eqref{hNLS-1} in the class
\begin{align*} 
  \Phi^{(c)}\in C(\R;L^2(\R^3;\C^N)) \cap L_t^{4}(\R;L^{3}(\R^3;\C^N)).
  \end{align*}  
Moreover, the global solution satisfies 
\begin{align}\label{Uniform Strichartz bound} 
\| \Phi^{(c)} \|_{ L_t^{4}([0,T];L^{3}(\R^3;\C^N)) } \ls (1+T)^{1/4},
\end{align}
where the implicit constant depends only on $\mathcal{M}(\Phi_0)$, but independent of $c$.
\end{proposition}

\begin{proof}
First, we prove the local well-posedness. Let $I=[0,T_0]$ be a sufficiently small interval to be chosen later. Duhamel's principle ensures that the solution to \eqref{hNLS-1} is equivalent to the following integral formula:
\begin{align*}
  \Gamma_k(\Phi^{(c)})=e^{-it\mathcal{H}^{c}_J(-\Delta)} \phi_{k,0} -i\int_0^te^{-i(t-s)\mathcal{H}^{c}_J(-\Delta)}\left(H(\phi_k^{(c)})-F_k(\phi_k^{(c)})\right)(s) \; ds.
\end{align*}
Let $B$ be a constant given by
\[B = 2A,\]
for $A$ as in \eqref{Strichartz for odd}. Then, the standard argument shows the map $\Gamma$ is contractive in the ball
\begin{align*} 
 X_{I,B}&=\Big\{ \Phi^{(c)} \in C_t(I;L^2(\R^3;\C^N)) \cap L_t^{4}(I;L^{3}(\R^3;\C^N)) :    \\
&\qquad\qquad \|\Phi^{(c)}\|_{X_I}:=\| \Phi^{(c)} \|_{  C_t(I;L^2(\R^3\;\C^N)) }
 +\| \Phi^{(c)} \|_{ L_t^{4}(I;L^{3}(\R^3;\C^N)) }
 \le 4B \| \Phi_0 \|_{ L^2(\R^3; \C^N) }
 \Big\}.
\end{align*}
Indeed, by unitarity and Strichartz estimates \eqref{Strichartz for odd}, we have 
\[\begin{aligned} 
&\| \Gamma(\Phi^{(c)}) \|_{C_t(I;L^2(\R^3;\C^N))} + \| \Gamma(\Phi^{(c)}) \|_{L_t^4(I;L^3(\R^3;\C^N))}\\
&\le 2B\| \Phi_0 \|_{ L^2(\R^3;\C^N) } 
\\
&~{}+ 2|\lambda|\sum_{k, \ell=1}^N
\left\{\| (|x|^{-1}\ast |\phi_{\ell}^{(c)}|^2)\phi_k^{(c)}\|_{ L_t^{1}(I;L^{2}(\R^3)) } 
+ \| (|x|^{-1}\ast (\overline{\phi_\ell^{(c)}}\phi_k^{(c)})\phi_{\ell}^{(c)} \|_{ L_t^{1}(I;L^{2})(\R^3) }  \right\}.
\end{aligned}\]
Then, by using  \eqref{nonlinear term in L2} and H\"older inequality in time, we conclude that 
\begin{align}\begin{aligned}\label{priori bound}
\| \Gamma(\Phi^{(c)}) \|_{X_I}
\le&~{} 2B\| \Phi_0\|_{ L^2(\R^3;\C^N) } \\
&~{}+ C T^{\frac12}\sum_{k, \ell =1}^N \|\phi_\ell^{(c)}\|_{L_t^{\infty}(I;L^{2}(\R^3))}\|\phi_\ell^{(c)}\|_{L_t^4(I;L^3(\R^3))}\|\phi_k^{(c)}\|_{L_t^4(I;L^3(\R^3))}\\
\le&~{} 2B\| \Phi_0\|_{ L^2(\R^3;\C^N) } + CT^{\frac12} \|\Phi^{(c)}\|_{X_I}^3.
\end{aligned}\end{align}
Analogously, we obtain such an a-priori bound for difference of two solutions $\Phi_1^{(c)}, \Phi_2^{(c)} \in X_{I,B}$ that
\begin{align}\label{contraction} 
\|\Gamma(\Phi_1^{(c)}) - \Gamma(\Phi_2^{(c)})\|_{X_I}  \ls 2CT^\frac12( \|\Phi_1^{(c)}\|_{X_I} + \|\Phi_2^{(c)}\|_{X_I} )^2\|\Phi_1^{(c)}-\Phi_2^{(c)}\|_{X_I}.
\end{align}
By taking an appropriate small $T_0$ depending on $\| \Phi_0\|_{ L^2(\R^3;\C^N) }$, the map $\Gamma$ is a contraction map in $X_{I,B}$, thus we can find the solution $\Phi^{(c)}$ satisfying 

\begin{align}\label{solution on I} 
 \|\Phi^{(c)} \|_{ X_I } \le 4B\| \Phi_0 \|_{ L^2(\R^3; \C^N) }. 
\end{align}
For any given $T>0$, using $\mathcal{M}(\Phi_0)$ and \eqref{solution on I} inductively, we show 
\begin{align}\label{time bound of solution}
  \| \Phi^{(c)}\|_{L_{t\in[0,T]}^4 L^{3}(\R^3; \C^N)}^4 \ls \sum_{j=1}^{[T/T_0]+1}  \| \Phi^{(c)}\|_{L_t^4(I_j;L^{3}(\R^3;\C^N))}^4 \lesssim T, 
 \end{align}
where $I_j=[(j-1)T_0,jT_0]$ and the implicit constant depends only on $\| \Phi_0 \|_{ L^2(\R^3; \C^N) }$, but not $c$.
\end{proof}


\subsection{Approximation : Proof of Theorem~\ref{Thm:Approximation}}
Before proving Theorem~\ref{Thm:Approximation}, we employ the following lemma.
\begin{lemma}\label{lem:5.4}
Let $c>1$ and $t \in \R$. For any $f \in H^{\frac12}(\R^3)$, we have 
\begin{align*} 
\| (e^{it(\sqrt{c^4-c^2\Delta}-c^2)}-e^{it\mathcal{H}^{c}_J(-\Delta)})f\|_{ L^2(\R^3) } \ls c^{-\frac{J}{2(J+1)}} \langle t \rangle \| f \|_{ H^\frac12(\R^3) }.
\end{align*}
\end{lemma}
\begin{proof}
  We write
  \begin{align*} 
    &(e^{it(\sqrt{c^4-c^2\Delta}-c^2)}-e^{it\mathcal{H}^{c}_J(-\Delta)})f \\
    &\quad =(e^{it(\sqrt{c^4-c^2\Delta}-c^2)}-e^{it\mathcal{H}^{c}_J(-\Delta)})P_{low}f
    +e^{it(\sqrt{c^4-c^2\Delta}-c^2)}P_{high}f+e^{it\mathcal{H}^{c}_J(-\Delta)}P_{high}f, 
  \end{align*}
 where $\widehat{P_{low}f}=\mathbf{1}_{|\xi|\le c^{\frac{J}{J+1}}}\widehat{f}$ and $P_{high}= 1-P_{low}$. By Taylor remainder theorem, we have 
\begin{align*} 
\Big| \sqrt{c^4+c^2|\xi|^2}-c^2-\sum_{j=1}^J \frac{(-1)^{j+1}(2j-2)!}{(j-1)!j!2^{2j-1}c^{2j-2}}|\xi|^{2j} \Big|  
\le \alpha_{J}c^{-2J}|\xi^*|^{2J+2},
\end{align*}
for some $0 < |\xi^*| < |\xi|$ and constant $\alpha_J$, which enables us to control the low frequency part as
 \begin{align*} 
&~{}\left\| (e^{it(\sqrt{c^4-c^2\Delta}-c^2)}-e^{it\mathcal{H}^{c}_J(-\Delta)})P_{low}f\right\|_{ L^2(\R^3) }\\
 \ls&~{} |t| \left\| \{\sqrt{c^4+c^2|\xi|^2}-c^2-\mathcal{H}^{c}_J(|\xi|^2)\} \widehat{f} \right\|_{L^2(|\xi|\le c^{\frac{J}{J+1}} )} \\
\ls&~{} \alpha_{J} c^{-2J}\| |\xi|^{2J+\frac32}|\xi|^{\frac12}  \widehat{\phi} \|_{L^2(|\xi|\le c^{\frac{J}{J+1}} )} 
  \ls \alpha_{J} c^{-\frac{J}{2(J+1)}} \|\phi\|_{H^\frac12}.
\end{align*}
On the other hand, it is not difficult to control high frequency parts:
\begin{align*} 
\|e^{it(\sqrt{c^4-c^2\Delta}-c^2)}P_{high} f\|_{L^2}
\ls \| |\xi|^{-\frac12}|\xi|^\frac12 \widehat{f} \|_{L^2(|\xi|\ge c^{\frac{J}{J+1}} )}
\ls c^{-\frac{J}{2(J+1)}} \|f\|_{H^\frac12}.
\end{align*}
The same argument can be applied to $e^{it\mathcal{H}^{c}_J(-\Delta)}P_{high}f$. Collecting all, we complete the proof.
\end{proof}

\begin{proof}[Proof of Theorem~\ref{Thm:Approximation}]
We write the difference of \eqref{hNLS} and \eqref{pNLHF} with $\Psi_0 = \Phi_0$ as
\[\begin{aligned} 
 \psi_k^{(c)}- \phi_k^{(c)} =&~{} (e^{-it(\sqrt{c^4-c^2\Delta}-c^2)}-e^{-it\mathcal{H}^{c}_J(-\Delta)})\psi_{k,0} \\
&~{}-i \int_{0}^{t} \left( e^{-i(t-s)(\sqrt{c^4-c^2\Delta}-c^2)} -e^{-i(t-s)\mathcal{H}^{c}_J(-\Delta)}\right) \left(H(\psi_k^{(c)})-F_k(\psi_k^{(c)})\right)(s)\; ds
(s)ds  \\
&~{}+i\kappa \int_{0}^{t} e^{-i(t-s)\mathcal{H}^{c}_J(-\Delta)}\mathcal{N}_1(\psi_k^{(c)},\phi_k^{(c)})(s)\; ds \\
&~{}+i\kappa \int_{0}^{t} e^{-i(t-s)\mathcal{H}^{c}_J(-\Delta)}\mathcal{N}_2(\psi_k^{(c)},\phi_k^{(c)})(s)\; ds \\
=:&~{} \mathcal{I} + \mathcal{II} + \mathcal{III} + \mathcal{IV},
\end{aligned}\]
for $k=1,2,\cdots,N$, where
\[\begin{aligned}
\mathcal{N}_1(\psi_k^{(c)},\phi_k^{(c)}) 
=&~{}  \sum_{\ell=1}^N \left(|x|^{-1}\ast(|\psi_{\ell}^{(c)}|^{2}  - |\phi_{\ell}^{(c)}|^{2})\right)\phi_k^{(c)} 
- \sum_{\ell=1}^N \left(|x|^{-1}\ast(\overline{\psi_{\ell}^{(c)}}\psi_k^{(c)} - \overline{\phi_{\ell}^{(c)}}\phi_k^{(c)})\right)\phi_{\ell}^{(c)},\\
\mathcal{N}_2(\psi_k^{(c)},\phi_k^{(c)}) 
=&~{} \sum_{\ell=1}^N \left(|x|^{-1}\ast(|\psi_{\ell}^{(c)}|^{2})\right)(\psi_k^{(c)} - \phi_k^{(c)}) 
- \sum_{\ell=1}^N \left(|x|^{-1}\ast( \overline{\psi_{\ell}^{(c)}}\psi_k^{(c)})\right)(\psi_{\ell}^{(c)} - \phi_{\ell}^{(c)}).
\end{aligned}\]
By Lemma \ref{lem:5.4}, we immediately obtain
\begin{equation}\label{eq:I}
\|\mathcal I\|_{L^2(\R^3)} \lesssim c^{-\frac{J}{J+1}}\langle t \rangle \|\psi_{k,0}\|_{H^{\frac12}(\R^3)}.
\end{equation}
Moreover, by using Lemma \ref{lem:5.4}, \eqref{Estimates for nonlinearity} and Proposition~\ref{Pro:HRNLS}, we have
\begin{equation}\label{eq:II}
\begin{aligned}
\|\mathcal{II}\|_{L^2(\R^3)} \lesssim&~{} c^{-\frac{J}{2(J+1)}}\int_0^t \langle t-s \rangle \left\|H(\psi_k^{(c)})(s)-F_k(\psi_k^{(c)})(s)\right\|_{H^{\frac12}(\R^3)}\; ds\\
\lesssim&~{}c^{-\frac{J}{2(J+1)}}\langle t \rangle^2\sup_{s\in[0,t]}\|\psi_k^{(c)}(s)\|_{ H^\frac12(\R^3)}^3 \ls c^{-\frac{J}{2(J+1)}}\langle t \rangle^2\|\psi_{k,0}\|_{H^{\frac12}(\R^3)}^3.
\end{aligned}
\end{equation}
Furthermore, by \eqref{nonlinear term in L2} and Sobolev embedding $H^\frac12(\R^3) \hookrightarrow  L^3(\R^3)$, and by \eqref{Sobolev inequality}, we have
\begin{equation}\label{eq:III} 
\begin{aligned}
\|\mathcal{III}\|_{L^2(\R^3)} \lesssim&~{} \int_0^t \left\|\mathcal{N}_1(\psi_k^{(c)},\phi_k^{(c)})(s) \right\|_{L^2(\R^3)}\; ds\\
\lesssim&~{}  \int_0^t \left(\|\Psi^{(c)}(s)\|_{H^{\frac12}(\R^3)} + \|\Phi^{(c)}(s)\|_{L^3(\R^3)}\right)\|(\Psi^{(c)} -\Phi^{(c)})(s)\|_{L^2(\R^3)}\\
&~{}\hspace{9em} \times \left(\|\psi_k^{(c)}(s)\|_{H^{\frac12}(\R^3)} + \|\phi_k^{(c)}(s)\|_{L^3(\R^3)}\right)\; ds\\
&~{}+ \int_0^t \left(\|\Psi^{(c)}(s)\|_{H^{\frac12}(\R^3)}^2 + \|\Phi^{(c)}(s)\|_{L^3(\R^3)}^2\right)\|(\psi_k^{(c)} -\phi_k^{(c)})(s)\|_{L^2(\R^3)}\; ds
\end{aligned}
\end{equation}
and
\begin{equation}\label{eq:IV}
\begin{aligned} 
\|\mathcal{IV}\|_{L^2(\R^3)} \lesssim&~{} \int_0^t \left\|\mathcal{N}_2(\psi_k^{(c)},\phi_k^{(c)})(s) \right\|\; ds\\
\lesssim&~{}  \int_0^t \|\Psi^{(c)}(s)\|_{H^{\frac12}(\R^3)}^2\|(\psi_k^{(c)} -\phi_k^{(c)})(s)\|_{L^2(\R^3)} \; ds\\
&~{}+ \int_0^t \|\Psi^{(c)}(s)\|_{H^{\frac12}(\R^3)}\|(\Psi^{(c)} -\Phi^{(c)})(s)\|_{L^2(\R^3)}\|\psi_k^{(c)}\|_{H^{\frac12}(\R^3)}\; ds
\end{aligned}
\end{equation}
respectively. Thus, by collecting \eqref{eq:I}--\eqref{eq:IV}, and by applying Proposition~\ref{Pro:HRNLS}, we conclude that
\[\begin{aligned}
&~{}\|\Psi^{(c)}(t)-\Phi^{(c)}(t)\|_{L^{2}(\R^3)} \\
\lesssim&~{} c^{-\frac{J}{2(J+1)}} \langle t \rangle^2\|\Psi_0\|_{H^{\frac12}(\R^3)}\left(1 + \|\Psi_0\|_{H^{\frac12}(\R^3)}^2\right)\\
&~{}+\int_0^t \left(\|\Psi_0(s)\|_{H^{\frac12}(\R^3)}^2 + \|\Phi^{(c)}(s)\|_{L^3(\R^3)}^2\right)\|(\Psi^{(c)} -\Phi^{(c)})(s)\|_{L^2(\R^3)}\; ds.
\end{aligned}\]
By Gronwall's inequality, in addition to \eqref{Uniform Strichartz bound},  we have
\[\|\Psi^{(c)}(t)-\Phi^{(c)}(t)\|_{L^{2}(\R^3)} \lesssim c^{-\frac{J}{2(J+1)}}Ae^{Bt},\]
here the constants $A, B$ depend on $\|\Psi_0\|_{H^{\frac12}}(\R^3)$, but not $c$.
\end{proof}


\section{Application~2: Small data scattering results}
Let $J \in 2\N -1$. Recall the higher-order nonlinear Schr\"odinger equations \eqref{gNLS}: 
\begin{equation}\label{gNLS-1}
  i\partial_t\psi_c= \mathcal{H}_J^{(c)}\psi_c
  +\kappa|\psi_c|^{\nu-1}\psi_c.
  \end{equation}
We first show the global well-posedness of \eqref{gNLS-1} in $H^1(\R^d)$. In what follows, we only focus on $d \ge 2$, since one-dimensional case can be easily proved by only using Sobolev embedding ($H^1(\R) \hookrightarrow L^{\infty}(\R)$). Let $r:=\nu +1$, and let fix $(q,r)$ as the admissible pair in \eqref{o admissible pair}. 
Let $0 < \epsilon \le \delta $ be sufficiently small to be chosen later. Then, for $\psi_{c,0} \in H^1(\R^d)$ with $\|\psi_{c,0}\|_{H^1(\R^d)} \le \epsilon$, the map
\begin{align}\label{integral equation}
\Gamma(\psi)=e^{-it\mathcal{H}_J^{(c)}(-\Delta)}\psi_{c,0} -i\kappa \int_0^te^{-i(t-s)\mathcal{H}_J^{(c)}(-\Delta)} \left(|\psi|^{\nu-1}\psi\right)(s) \; ds
\end{align}
is contractive in the ball
\begin{align*}
X_{\delta}:=\{\psi: \R\times\R^d\rightarrow \C  : \|\psi\|_{X^{1}}:=\| \psi\|_{L_t^\infty(\R ;  H^1(\R^d))} +\|\psi\|_{L_t^q(\R; W^{1,r}(\R^d))} \le \delta \}
\end{align*}
with the topology induced by the metric\footnote{We refer to \cite{Cazenavebook} (in the proof of Theorem 4.4.1) for verifying that $(L^{\infty}H^1 \cap L^qW^{1,r}, d)$ is a complete metric space.}
\[d(u,v) = \|u-v\|_{L_t^{\infty}(\R; L^2(\R^d))} + \|u-v\|_{L_t^q(\R; L^r(\R^d))}.\]
By Strichartz estimates \eqref{Strichartz for odd} and \eqref{Inhomogeneous Strichartz for odd}, we have that 
\begin{align*}
\|\Gamma(\psi)\|_{X^1} &\le  A\| \psi_{c,0}\|_{H^1(\R^d)} + C \| |\psi|^{\nu-1}\psi\|_{L_t^{q'}(\R;W^{1,r'}(\R^d))},
\end{align*}
for $A$ as in \eqref{Strichartz for odd} and some constant $C > 0$, independent of $c$ and $J$, where $(q',r')$ is the H\"older conjugate of $(q,r)$. Note that
\begin{equation}\label{eq:gradient}
\nabla(|\psi|^{\nu -1}\psi) = \left(1 + \frac{\nu -1}{2}\right)|\psi|^{\nu -1} \nabla \psi + \frac{\nu -1}{2}|\psi|^{\nu-3}\psi^2 \nabla\overline{\psi}.
\end{equation}
Using Hölder inequality, we immediately obtain
\begin{equation}\label{eq:L^r' control}
\| |\psi|^{\nu-1}\psi\|_{L^{r'}(\R^d)} \le \| |\psi|^{\nu-1}\|_{L^{\tilde r}(\R^d)}\|\psi\|_{L^r(\R^d)} \le \|\psi\|_{L^r(\R^d)}^{\nu},
\end{equation}
where
\begin{equation}\label{rtilde}
\frac{1}{\tilde{r}} = \frac{1}{r'} - \frac{1}{r} = \frac{\nu -1}{\nu +1} \quad \mbox{which gives}\;\; \tilde{r}(\nu - 1) = \nu +1 =r.
\end{equation}
Analogously, we have from \eqref{eq:gradient} that
\[ \|\nabla( |\psi|^{\nu-1}\psi)\|_{L^{r'}(\R^d)} \lesssim \|\psi\|_{L^r(\R^d)}^{\nu-1}\|\nabla \psi\|_{L^r(\R^d)}.\]
Collecting all, we have
\begin{equation}\label{eq:a priori bound}
 \| |\psi|^{\nu-1}\psi\|_{L_t^{q'}(\R;W^{1,r'}(\R^d))} \lesssim \|\psi\|_{L_t^{\tilde{q}(\nu-1)}(\R; L^r(\R^d))}^{\nu-1}\|\psi\|_{L_t^q(\R; W^{1,r}(\R^d))},
 \end{equation}
where $\tilde{q}$ satisfies under the condition \eqref{H1 LWP} that
\begin{equation}\label{qtilde}
\frac{1}{\tilde{q}} = \frac{1}{q'} - \frac{1}{q} = \frac{(2-d)\nu +2 + d}{2(\nu +1)} > 0 \quad \mbox{which gives}\;\; q < \tilde{q}(\nu - 1).
\end{equation}
%
Under the condition $q < \tilde{q}(\nu - 1)$, a direct computation, in addition to Sobolev embedding ($H^1(\R^d) \hookrightarrow L^r(\R^d)$), yields
\begin{equation}\label{a priori bound}
\begin{aligned}
\| \psi\|_{L_t^{\tilde q(\nu-1)}(\R;L^r(\R^d))}^{\nu-1} \lesssim&~{} \left(\int_\R \|\psi(t)\|_{H^1(\R^d)}^{\tilde{q}(\nu-1) - q}\|\psi(t)\|_{L^r(\R^d)}^{q} \; dt \right)^{\frac{1}{\tilde q}}\\
\ls&~{} \| \psi\|_{L_t^\infty (\R;H^1(\R^d))}^{\frac{\tilde q(\nu-1)-q}{\tilde q}} \|\psi\|_{L_t^q(\R; L^r(\R^d))}^{\frac{q}{\tilde q}} \\
\lesssim &~{} \|\psi\|_{X^1}^{\nu-1},
\end{aligned}
\end{equation}
thus we conclude that
\begin{align*}
\|\Gamma(\psi)\|_{X^1} &\le  A\| \psi_{c,0}\|_{H^1(\R^d)} + C_1 \| \psi\|_{X^1}^{\nu},
\end{align*}
for some constant $C_1 > 0$. By taking $\epsilon, \delta > 0$ satisfying
\begin{equation}\label{eq:well-define}
A\epsilon + C_1 \delta^{\nu} < \delta,
\end{equation}
we claim that the map $\Gamma$ is well-defined from $X_{\delta}$ to itself. 

Now, we show the map $\Gamma$ is contractive with respect to the metric $d$. Let $\psi_1$ and $\psi_2$ satisfy \eqref{integral equation} with $\psi_1(0,x) = \psi_2(0,x)$. A direct computation gives
\[|\psi_1|^{\nu-1}\psi_1 - |\psi_2|^{\nu-1}\psi_2 = |\psi_1|^{\nu-1}(\psi_1 - \psi_2) + (|\psi_1|^{\nu-1} - |\psi_2|^{\nu-1})\psi_2,\]
which, in addition to the Mean Value theorem for complex-valued functions, implies
\[\left||\psi_1|^{\nu-1}\psi_1 - |\psi_2|^{\nu-1}\psi_2\right| \lesssim (|\psi_1|^{\nu-1} + |\psi_2|^{\nu-1})|\psi_1 - \psi_2|.\]
Then, similarly as in \eqref{eq:L^r' control} and \eqref{a priori bound}, we have
\[\begin{aligned}
d\left(\Gamma(\psi_1),\Gamma(\psi_2)\right) =&~{}\|\Gamma(\psi_1)-\Gamma(\psi_2)\|_{L_t^{\infty}(\R; L^2(\R^d))} + \|\Gamma(\psi_1)-\Gamma(\psi_2)\|_{L_t^q(\R; L^r(\R^d))}\\
\lesssim&~{} \|(|\psi_1|^{\nu-1} + |\psi_2|^{\nu-1})|\psi_1 - \psi_2|\|_{L_t^{q'}(\R; L^{r'}(\R^d))}\\
\lesssim&~{}(\|\psi_1\|_{L_t^{\tilde q(\nu-1)}(\R; L^r(\R^d))}^{\nu-1} + \|\psi_2\|_{L_t^{\tilde q(\nu-1)}(\R; L^r(\R^d))}^{\nu-1})\|\psi_1 - \psi_2\|_{L_t^q(\R; L^r(\R^d))}\\
\le&~{}C_2(\|\psi_1\|_{X^1}^{\nu-1} + \|\psi_2\|_{X^1}^{\nu-1})\|\psi_1 - \psi_2\|_{L_t^q(\R; L^r(\R^d))}\\
\le&~{}2C_2\delta^{\nu-1} d(\psi_1,\psi_2).
\end{aligned}\]
By taking $\delta > 0$ satisfying $2C_2 \delta^{\nu-1} < \frac12$ and \eqref{eq:well-define},
we complete the proof of the global well-posedness of \eqref{gNLS-1} in $H^1(\R^d)$\footnote{The continuous dependence on initial data follows from the analogous way.}. 

Now, we prove that the solution $\psi_c(t)\in C(\R;H^1(\R^d))$ to \eqref{integral equation}, that we found above, scatters to a free solution as $t\rightarrow\infty$\footnote{An analogous way also allows the same conclusion when $t \to -\infty$.}.  Let us define scattering state $\psi_{J,+}^{(c)}$ by the formula
\begin{align*}
\psi_{J,+}^{(c)}:=\psi_{c,0}- i \kappa \int_0^\infty e^{-is\mathcal{H}_J^{(c)}(-\Delta)} \left(|\psi_c|^{\nu-1}\psi_c\right)(s) \; ds.
\end{align*}
Then, an analogous argument as in the proof of \eqref{eq:a priori bound} with \eqref{a priori bound}, we show
\begin{align*}
\left\| \int_{t_1}^{t_2}e^{-is\mathcal{H}_J^{(c)}(-\Delta)} \left(|\psi_c|^{\nu-1}\psi_c\right)(s) \; ds \right\|_{H^1(\R^d)} \ls&~{}  \big\| |\psi_c|^{\nu-1}\psi_c \big\|_{L_t^{q'}([t_1, t_2); W^{1,r'}(\R^d))}\\
 \ls&~{}   \| \psi_c\|_{L_t^\infty([t_1,t_2); H_x^1(\R^d)}^{\frac{\tilde q(\nu-1)-q}{\tilde q}} \|\psi_c\|_{L_t^q([t_1,t_2); W_x^{1,r}(\R^d)}^{1+ \frac{q}{\tilde q}},
\end{align*}
which says $\psi_{J,+}^{(c)} \in H^1(\R^d)$, and the left-hand side tends to $0$ as $t_1,t_2\rightarrow\infty$.
This immediately shows
\begin{align*}
\left\|\psi_c(t) - e^{-it\mathcal{H}_J^{(c)}(-\Delta)}\psi_{J,+}^{(c)}\right\|_{H^1(\R^d)}
=\left\| \int_t^\infty e^{-i(t-s)\mathcal{H}_J^{(c)}(-\Delta)} \left(|\psi_c|^{\nu-1}\psi_c\right)(s) \; ds \right\|_{H^1(\R^d)} \rightarrow 0
\end{align*}
as $t\rightarrow\infty$, which completes the proof.





  \providecommand{\bysame}{\leavevmode\hbox to3em{\hrulefill}\thinspace}
\providecommand{\MR}{\relax\ifhmode\unskip\space\fi MR }
\providecommand{\MRhref}[2]{%
  \href{http://www.ams.org/mathscinet-getitem?mr=#1}{#2}
}
\providecommand{\href}[2]{#2}

\end{document}